\documentclass[12pt,english]{article}

\usepackage{hyperref}
\hypersetup{
    unicode = false,
    pdftoolbar = true,
    pdfmenubar = true,
    pdffitwindow = true,
    pdftitle = {StoMADS},
    pdfauthor = {DZAHINI},
    pdfsubject = {StoMADS},
    pdfnewwindow = true,
    pdfkeywords = {Stochastic, MADS},
    colorlinks = true,
    linkcolor = blue,
    citecolor = blue,
    filecolor = black,
    urlcolor = blue,
    breaklinks = true
}

\usepackage[english]{babel}
\usepackage[T1]{fontenc}
\usepackage[latin1]{inputenc}
\usepackage{amsfonts}
\usepackage{dsfont}
\usepackage{amsmath}
\usepackage{amsthm}
\usepackage{amssymb}
\usepackage{multirow}
\usepackage{epsfig}
\usepackage{soul}
\usepackage{colortbl}
\usepackage{geometry}
\usepackage{times}
\usepackage{graphicx}
\usepackage{mathrsfs}
\usepackage{frcursive}
\usepackage[ruled, vlined, linesnumbered]{algorithm2e}
\usepackage{color}
\usepackage{booktabs}

\definecolor{gris3}{rgb}{0.3,0.3,0.3}

\definecolor{Green}{rgb}{0,.6,0}
\definecolor{Blue}{rgb}{0,0,1}
\definecolor{Red}{rgb}{1,0,0}
\definecolor{Gray}{rgb}{0.2,0.2,0.2}
\definecolor{Maroon}{rgb}{0.6,0.05,0.03}

\newcommand{\blc}{\color{black}}

\newcommand{\blu}{\color{black}}
\newcommand{\blr}{\color{black}}

\usepackage{tikz}
\usetikzlibrary{shapes}
\usetikzlibrary{fit, arrows, calc, positioning}
\usetikzlibrary{arrows, automata}
\usetikzlibrary{positioning}
\usetikzlibrary{shadows}

\tikzstyle{c} = [rectangle, rounded corners, draw]
\tikzset{box/.style={draw,rectangle,rounded  corners=0pt, fill=gray!15, minimum  width=3cm, minimum  height=2cm}}
\tikzset{loz/.style={draw,diamond, aspect=1.5, fill=gray!15}}

\tikzstyle{rnd}=[circle,fill=blue!25,minimum  width=5em]

\newtheorem{theorem}{Theorem}  
\newtheorem{proposition}{Proposition}

\newtheorem{corollary}{Corollary}
\newtheorem{lemma}{Lemma}
\newtheorem{remark}{Remark}

\newtheorem{assumption}{Assumption}

\newtheorem{definition}{Definition}

\newcommand{\Z}{\mathbb{Z}}
\newcommand{\R}{\mathbb{R}}
\newcommand{\Esp}{\mathbb{E}}
\newcommand{\rn}{\mathbb{R}^n}
\newcommand{\rnn}{\mathbb{R}^{n\times n}}

\newcommand{\N}{\mathbb{N}}

\newcommand{\lun}{\mathbb{L}^1(\Omega,\mathcal{G},\mathbb{P})}
\newcommand{\lp}{{\mathbb{L}}^p(\Omega,\mathcal{G},\mathbb{P})}
\newcommand{\ldeux}{{\mathbb{L}}^2(\Omega,\mathcal{G},\mathbb{P})}

\newcommand{\M}{\mathcal{M}}

\newcommand{\dpl}{\delta^k_p}
\newcommand{\Dp}{\Delta^k_p}
\newcommand{\Dm}{\Delta^k_m}
\newcommand{\dm}{\delta^k_m}

\newcommand{\fd}{f}
\newcommand{\fe}{f_\Theta}
\newcommand{\Mk}{\mathcal{M}^k}
\newcommand{\fok}{f^k_0}
\newcommand{\fsk}{f^k_s}
\newcommand{\dmax}{\tau^{-\hat{z}}}
\newcommand{\Fok}{F^k_0}
\newcommand{\Fsk}{F^k_s}
\newcommand{\ef}{\varepsilon_f}
   
\newcommand{\isucc}{\mathds{1}_{S}}   

\newcommand{\iunsuc}{\mathds{1}_{{\bar{S}}^C}}
\newcommand{\iunc}{\mathds{1}_{{\bar{S}}^{\bar{C}}}}
\newcommand{\ijk}{\mathds{1}_{J_k}}
\newcommand{\iji}{\mathds{1}_{J_i}}
\newcommand{\ijkc}{\mathds{1}_{\bar{J_k}}}

\newcommand{\isbar}{\mathds{1}_{\bar{S}}} 
\newcommand{\abs}[1]{\left\lvert#1\right\rvert}
\newcommand{\E}[1]{\mathbb{E}\left(#1\right)}
\newcommand{\pr}[1]{\mathbb{P}\left(#1\right)}

\newcommand{\norme}[1]{\left\lVert#1\right\rVert}

\newcommand{\normp}[1]{{\left\lVert#1\right\rVert}_{p}}

\newcommand{\norminf}[1]{{\left\lVert#1\right\rVert}_{\infty}}

\title{
    {\blc StoMADS: S}tochastic {\blc b}lackbox {\blc o}ptimization using {\blc p}robabilistic {\blc e}stimates
}

\author{
	\href{mailto:charles.audet@gerad.ca}{Charles Audet}\thanks{GERAD and DÈpartement de MathÈmatiques et de GÈnie Industriel, \'Ecole Polytechnique de MontrÈal, C.P. 6079, Succ. Centre-ville, MontrÈal, QuÈbec H3C 3A7, Canada (\href{mailto:charles.audet@gerad.ca}{Charles.Audet@gerad.ca}, 
		\href{https://www.gerad.ca/Charles.Audet}{www.gerad.ca/Charles.Audet}, \href{mailto:kwassi-joseph.dzahini@polymtl.ca}{Kwassi-Joseph.Dzahini@polymtl.ca}, \href{https://www.gerad.ca/fr/people/kwassi-joseph-dzahini}{www.gerad.ca/fr/people/kwassi-joseph-dzahini}, 	\href{mailto:sebastien.le.digabel@gerad.ca}{Sebastien.Le.Digabel@gerad.ca}, \href{https://www.gerad.ca/Sebastien.Le.Digabel}{www.gerad.ca/Sebastien.Le.Digabel}).
	}
	\and
	\href{mailto:kwassi-joseph.dzahini@polymtl.ca}{Kwassi Joseph Dzahini}\footnotemark[1]	
	\and
	\href{mailto:michael.kokkolaras@mcgill.ca}{Michael Kokkolaras}\thanks{GERAD and McGill University, Mechanical Engineering Department, 845 Rue Sherbrooke Ouest, MontrÈal, QuÈbec H3A 0G4, Canada (\href{mailto:michael.kokkolaras@mcgill.ca}{Michael.Kokkolaras@mcgill.ca}, 
		\href{https://www.mcgill.ca/mecheng/people/staff/michael-kokkolaras}{www.mcgill.ca/mecheng/people/staff/michael-kokkolaras}).
	}
	\and
	\href{mailto:Sebastien.Le.Digabel@gerad.ca}{S\'ebastien Le Digabel}\footnotemark[1]
}

\begin{document}
\maketitle

\vspace*{-0.5cm}

\noindent
{\bf Abstract:}
This work introduces {\blc StoMADS}, a stochastic variant of {\blc the {\blc m}esh {\blc a}daptive {\blc d}irect-{\blc s}earch (MADS) algorithm {\blu originally developed} for deterministic blackbox optimization. StoMADS} considers the {\blc unconstrained} optimization of an objective function $f$ whose values can be computed {\blu only} through {\blu a blackbox corrupted by} some random noise {\blc following} an unknown distribution. The proposed {\blc method is based on} an algorithmic {\blc framework} similar to that of MADS and {\blc uses} random estimates of function values obtained from stochastic observations since the exact deterministic computable version of $f$ is not available. {\blc Such estimates are required to be accurate with a sufficiently large but fixed probability and satisfy a variance condition. The ability of the proposed algorithm to generate an asymptotically dense set of search directions is then exploited to show convergence to a Clarke stationary point {\blc of $f$} with probability~one, using martingale theory.}

\noindent
{\bf \blc Key words:} Blackbox optimization, Derivative-free optimization, Stochastic optimization, Mesh adaptive direct-search, Probabilistic estimates.


%

\clearpage
\clearpage
\newpage
\section{Introduction} 

Blackbox optimization (BBO), {\blr an} instance of derivative-free optimization {\blc (DFO)}, is the study of design and analysis of algorithms that assume {\blu that} the objective and/or constraint functions are {\blc provided} by {\blu blackboxes, i.e, ``{\it any processes whose inner workings are analytically unavailable and which return an output, when provided an input}''}~\cite{AuHa2017}. 

This work introduces a stochastic variant of the {\blc m}esh {\blc a}daptive {\blc d}irect {\blc s}earch (MADS) algorithm~\cite{AuDe2006} {\blc for deterministic BBO} and analyze{\blc s} it using {\blc elements from}~\cite{AuDe2006,BaScVi2014,chen2018stochastic,paquette2018stochastic}. {\blc It aims to solve} the following stochastic {\blc blackbox} optimization problem:

\begin{equation}\label{problem1}
{\blr \underset{x\in \R^n}{\min}\   f(x) \quad \text{where}\quad f(x)=\Esp_\Theta\left[f_\Theta(x)\right]}
\end{equation}
$\Theta$ is a random variable obeying some unknown distribution, {\blc $\Esp_\Theta$ denotes the expectation with respect to $\Theta$}, {\blc $f_\Theta$} denotes the {\blc blackbox, the} noisy computable version of the numerically unavailable {\blc objective} function {\blc $f\colon\rn\mapsto \R$.} 
{\blc In the convergence analysis {\blc of Section~\ref{convAnalysis}}, the objective function} is assumed to be {\blc locally Lipschitz} continuous and bounded from below.  

{\blc Such problems are of utmost importance and often arise in modern statistical machine learning, where the random variable $\Theta$ represents a data point drawn according to some unknown distribution and $f_{\Theta}(x)$ measures the fit of some model parameter $x$ to the data point $\Theta$} ~\cite{balasubramanian2019zeroth,curtis2017stochastic,kulunchakov2019estimate}.

{\blu The study of these problems and specifically, developing provable algorithms to solve~\eqref{problem1}, has been a topic of intense research. In the recent years, several methods have been developed, most of which are extensions of existing traditional deterministic DFO methods~\cite{AuHa2017,CoScVibook} to stochastic functions}~\cite{augustin2017trust,Ch2012,chen2018stochastic,LaBi2016,paquette2018stochastic,shashaani2018astro}. Such methods {\blc are} classified  according to Ang¸n and Kleijnen~\cite{angun2012asymptotic} into two categories~\cite{Ch2012}: {\blr White-box} and blackbox methods. {\blc White-box methods are} those where one has the ability to {\blc carry out an estimation of} the gradient {\blc of $f$} {\blc by means of a single simulation}{\blc . P}erturbation analysis~\cite{Ch2012} and {\blc the} likelihood ratio function method~\cite{fu2006gradient} {\blc being some examples among many others}. {\blc Blackbox methods are those who essentially process the simulation model as a blackbox}, {\blc such as the} 
{\blc s}tochastic {\blc a}pproximation method~\cite{kiefer1952stochastic}, {\blc r}esponse {\blc s}urface {\blc m}ethodology~\cite{amaran2014simulation}{\blc , and} many heuristics~\cite{amaran2014simulation}. Thorough descriptions of {\blc stochastic approximation and response surface methodology} are provided in~\cite{AnFe01a}. 

However, in many real applications, {\blu the} simulation model {\blu is inaccessible}~\cite{Ch2012} or the estimation of the gradient can be computationally expensive{\blr .} Direct-search blackbox optimization methods, generally known to be robust and reliable in practice~\cite{Audet2014a}, appear to be the most {\blu promising} option. It is {\blu important to emphasize} that the analysis in {\blc the present} work {\blc does not assume the} {\blc existence} of derivatives, i.e, first-order information{\blu ,} and consequently no gradient approximations will be carried out. 

Examples of existing traditional deterministic direct-search blackbox optimization method that have been extended to stochastic functions include the Nelder-Mead (NM) method~\cite{NeMe65a}. After Barton and Ivey~\cite{barton1996nelder}{\blu ,} who are among the first authors to propose a variant of the NM algorithm designed to cope with noisy function evaluations, Anderson and Ferris~\cite{AnFe01a} also considered {\blc the} unconstrained optimization of functions with evaluations subject to a random noise. They {\blu used} an algorithmic framework similar to that of NM, {\blc making use of so-called} {\it structures} instead of {\blu simplices} and {\blc propose} an algorithm involving reflection, expansion and contraction steps, which is shown to converge to a point with probability one, based on Markov {\blc c}hains theory~\cite{durrett2010probability}. Chang~\cite{Ch2012} proposed a new variant of the classic NM method, the {\blc s}tochastic Nelder-Mead method. After replacing the shrink step of the classic NM by {\blc the {\it {\blc a}daptive {\blc r}andom {\blc s}earch}, which is a local and global search framework}, in order to {\blc avoid a precocious convergence of the new algorithm}, he {\blu proved} convergence of {\blc the stochastic Nelder-Mead method} to global optima with probability one. 

Audet et al.~\cite{AudIhaLedTrib2016} recently proposed Robust-MADS, a kernel smoothing-based variant of the MADS~\cite{AuDe2006} algorithm designed to approach the minimizer of an objective function when only having access to noisy function values. At each iteration {\blc of Robust-MADS}, {\blc the incumbent solution} is determined based on values of the  smoothed version of the noisy available objective constructed {\blc from a list of trial points.} This list is then {\blc eventually updated} with the best iterate found before the next iteration of the algorithm. The proposed method is shown to have zero-order~\cite{AuDeLe07} convergence properties: Iterates produced by Robust-MADS converge to a point which is ``{\blc \em the limit of mesh local optimizers on meshes that get infinitely fine}''{\blc~\cite{AuDe2006}}. Note however that even though this method {\blc produces} interesting results when applied to problems {\blc including those involving granular and discrete variables}~\cite{AuLeDTr2018}, the corresponding work present{\blc s} no {\blc computational} tests to show how the proposed algorithm behaves on problems involving random noise, {\blc i.e, in a stochastic framework}. Furthermore, Robust-MADS results in a deterministic algorithm in the sense that it uses only deterministic algorithmic objects, i.e, mesh and frame size parameters, smoothed function values, {\blc etc.} to ensure improvements, in such a way that the resulting convergence of algorithm iterates should be understood  from a deterministic and non-stochastic angle. 

Moreover, note that unlike the present research where the noise distribution is {\blu assumed} to be unknown,~\cite{PierreYvesBouchet} considers the optimization of functions that are numerically unavailable and whose values can only be computed through a blackbox corrupted {\blu by} Gaussian random noise. Using an algorithmic framework similar to that of MADS, the algorithm proposed in~\cite{PierreYvesBouchet} aims to minimize such unknown functions by adaptively driving to zero the standard deviation of the estimators of the unavailable function values, making use of statistical inference techniques. However, even though this algorithm is shown to have desirable convergence properties, it needs to be improved since obtaining satisfactory solutions in practice requires a lot of blackbox evaluations, thus making the method {\blu computationally expensive}.

This study proposes StoMADS, a stochastic variant {\blc of MADS}, designed to cope with the {\blc unconstrained} optimization of stochastic blackbox functions {\blu while guaranteeing convergence to a Clarke stationary point  provided that certain conditions are satisfied}. The proposed work uses an algorithmic framework similar to that of MADS in addition to assumptions including those taken from~\cite{chen2018stochastic,paquette2018stochastic}. More precisely, it has been assumed that function estimates that are used to ensure improvements in the algorithm need to be {\blc accurate enough} with a fixed probability which does not have to equal one but simply needs to be above a certain constant~\cite{chen2018stochastic,paquette2018stochastic}{\blr .} {\blr I}n addition to the fact that such estimates are further assumed to satisfy a variance condition~\cite{paquette2018stochastic} that will be specified later, {\blu no assumption is made} about their distribution {\blc nor about the way they are generated}. 

The main novelty of {\blr the present} work is that no model or gradient information is needed to find descent directions, compared to prior works, in particular~\cite{chen2018stochastic,paquette2018stochastic} and~\cite{wang2019stochastic}. {\blu This} work uses direct-search techniques and then exploits the ability of the proposed algorithm to generate an asymptotically dense set of search directions to guarantee convergence. To the best of our knowledge, this {\blc research} {\blc is the first} to propose a stochastic variant of MADS with full-supported convergence results, obtained {\blc using} martingale theory.

This {\blc manuscript} is organized as follows. Section~\ref{sec2} introduces the general framework of the proposed stochastic method and discusses the requirements on random estimates to guarantee convergence in addition to how such estimates can be obtained in practice. It is followed by Section~\ref{convAnalysis} which presents the main convergence results. Computational results {\blc are reported} in {\blc Section~\ref{sec4}}, followed by {\blc a} discussion and {\blu suggestions for} future work.
\section{The StoMADS algorithm and probabilistic estimates}\label{sec2}

{\blu This section presents the general framework of StoMADS and introduces random quantities such as probabilistic estimates that {\blr are} useful for the convergence analysis. It then shows how such estimates can be constructed.}
\subsection{{\blc The} StoMADS algorithm}\label{stoalg}
{\blc Similarly to} MADS~\cite{AuDe2006}, StoMADS is an iterative algorithm where each iteration is characterized by two main steps: an optional SEARCH step which consists of a global {\blc exploration} that {\blc may} use various strategies including the use of surrogate functions and heuristics, to explore the variables space, and a local POLL step which follows stricter rules and performs a local {\blc exploration} in {\blc a subset of} the space of variables, {\blc called the {\it frame}}. During each of these two steps, a finite number of trial points are generated on a discretization of the space of variables called the {\it mesh}. {\blc The discretization of the mesh and frame is controlled by the mesh and frame size parameters, $\dm$ and $\dpl$, respectively, thus disparting from the notation $\delta^k$ and $\Delta^k$ from~\cite{AuHa2017} because $\Dm$ and $\Dp$ will be used to denote random variables.}

{\blc Let $\mathbf{D}\in\R^{n\times p}$ be a matrix, whose columns denoted by the set $\mathbb{D}$ form a positive spanning set. The mesh $\M^k$ and the frame $\mathcal{F}^k$ are respectively
\[\mathcal{M}^k:= \{x^k+\dm d: d=\mathbf{D}y,\ y\in \Z^p\}\quad\text{and}\quad \mathcal{F}^k := \{x\in \mathcal{M}^k, \norminf{x-x^k}\leq \dpl b \}, \]
where $b=\max\{\norminf{d'}, d'\in \mathbb{D}\}$.}

At iteration $k$, given an incumbent solution $x^k\in\Mk$, the {\blc StoMADS} algorithm seeks to find a trial {\blc {\it ``improved mesh point''}~\cite{AuDe2006}}  $y=x^k+\dm d$ whose objective function value is {\blc less} than the current unknown incumbent value $\fd(x^k)$, i.e $\fd(y)<\fd(x^k)$. {\blc In the present work, $\fok$ and $\fsk$ denote respectively the estimates of $\fd(x^k)$ and $\fd(x^k+s^k)$, where $s^k=\dm d$}{\blu ,} constructed using {\blc evaluations} of the available noisy blackbox $f_\Theta$. Such estimates are then compared {\blc in a way specified} below, to determine whether a trial point $x^k+s^k$ may be an improved mesh point or not. 

In both the SEARCH and POLL steps, unlike {\blc the} MADS algorithm where function values $\fd(x^k)$ and $\fd(x^k+s^k)$ are available, informations provided by {\blc the} estimates $\fok$ and $\fsk$ are used to determine whether a trial point $x^k+s^k$ may be an improved mesh point or not, i.e.{\blu ,} whether an iteration is successful or not. Thus, such estimates {\blc need to be sufficiently accurate}. {\blc The following definition is adapted from}~\cite{chen2018stochastic}.
\begin{definition}\label{probestim1}
{\blc	Let $\ef>0$ be a fixed constant and $f_{x}$ be an estimate of $\fd(x)$. Then $f_{x}$ is said to be an $\ef$-accurate estimate of $\fd(x)$  for a given $\dpl${\blc ,} if}
	\begin{equation}\label{estim1}
	\abs{f_{x}-\fd(x)}\leq \ef(\dpl)^2.\nonumber
	\end{equation}
\end{definition}

Note that{\blu ,} unlike~\cite{chen2018stochastic,wang2019stochastic}, $\ef$ does not play a crucial role in the convergence analysis but allows to adjust the initial amplitude of the so-called uncertainty interval $\mathcal{I}_{\gamma,\ef}(\dpl)$ that will be introduced later. {\blc The next result provides sufficient information to determine the iteration type.}
\begin{proposition}\label{decrease1}
	Let $f^k_0$ and $f^k_s$ be $\ef$-accurate estimates of $\fd(x^k)$ and $\fd(x^k+s^k)$, respectively, and let $\gamma\in (2,+\infty)$ be a fixed constant. Then the followings hold:
\begin{eqnarray*}
		&\text{{\blc if}}&  \fsk -\fok \leq -\gamma\ef(\dpl)^2,\ \ \text{{\blc then}}\ \ \fd(x^k+s^k)-\fd(x^k)<0,\\
		\text{and}\!\!\! &\text{{\blc if}}&  \fsk -\fok \geq \gamma\ef(\dpl)^2,\quad \ \text{{\blc then}}\ \ \fd(x^k+s^k)-\fd(x^k)>0.
\end{eqnarray*}
\end{proposition}

\begin{proof}
	The proof is immediate using Definition~\ref{probestim1} and {\blc observing} that
	\[\fd(x^k+s^k)-\fd(x^k)=\fd(x^k+s^k)-\fsk+\left(\fsk-\fok\right)+\fok-\fd(x^k). \] 
\end{proof}

{\blc The following definition distinguishes three types of iterations: successful, certain unsuccessful and uncertain unsuccessful.}
\begin{definition}\label{decreasedef}
Let $f^k_0$ and $f^k_s$ be $\ef$-accurate estimates of $\fd(x^k)$ and $\fd(x^k+s^k)$, respectively, and let $\gamma\in (2,+\infty)$ be a fixed constant. Then the iteration is called:
\[
\left\{
\begin{array}{ll}
\text{successful} & \mbox{if }\ \ \fsk -\fok \leq -\gamma\ef(\dpl)^2, \\
\text{unsuccessful and certain} & \mbox{if }\ \ \fsk -\fok \geq \gamma\ef(\dpl)^2, \\
\text{unsuccessful and uncertain} & \mbox{if }\ \ \fsk -\fok\in \mathcal{I}_{\gamma,\ef}(\dpl):= \left]-\gamma\ef(\dpl)^2,\gamma\ef(\dpl)^2\right[
\end{array}
\right. 
\]
where $\mathcal{I}_{\gamma,\ef}(\dpl)$ is the so-called {\blc \it uncertainty interval that is reduced during uncertain unsuccessful iterations}.
\end{definition}
{\blc Let $\tau\in (0,1)\cap\mathbb{Q}$ be a fixed constant} {\blc and $\hat{z}\in \N$  be a large fixed integer}.
{\blc Note that for the needs of the convergence analysis of Section~\ref{convAnalysis}, unlike MADS, the frame size parameter of StoMADS is supposed to be bounded above by a positive fixed constant $\dmax$ in order for the random frame size parameter $\Dp$ that will be introduced in the next subsection to be integrable. }

During the SEARCH or POLL step, {\blc if} the {\blc \it sufficient decrease condition} $\fsk -\fok \leq -\gamma\ef(\dpl)^2$ is {\blc satisfied} for some direction $s^k=\dm d$, then the iterate $x^k+s^k$ is successful according to Proposition~\ref{decrease1}. Hence, the current iterate and the frame size parameter are updated respectively according to $x^{k+1}=x^k+s^k$ and  {\blc $\delta_p^{k+1}=\min\{\tau^{-2}\dpl,\dmax\}$}, and then a new iteration is initiated with a new mesh size parameter $\delta_m^{k+1}$ which satisfies $\delta_m^{k+1}=\min\{\delta_p^{k+1},(\delta_p^{k+1})^2\} $. 

If no improved mesh point is found during the SEARCH step, then the POLL step is invoked and if the condition  $\fsk -\fok \leq -\gamma\ef(\dpl)^2$ does not hold, the iterate is unsuccessful according to Proposition~\ref{decrease1}. StoMADS presents two types of unsuccessful iterations: {\it certain unsuccessful iterations} and {\it uncertain unsuccessful iterations}.
In both {\it certain} and {\it uncertain unsuccessful iterations}, the current iterate is not updated, i.e $x^{k+1}=x^k$ and the corresponding frame $\mathcal{F}^k$ is said to be a {\it minimal frame} with {\it minimal frame center} $x^k$, also called a {\it mesh local optimizer}~\cite{AudIhaLedTrib2016}. However, one may notice that if the unsuccessful iteration is certain, then the frame size parameter is reduced according to $\delta_p^{k+1}=\tau^{2}\dpl$ {\blc so that the resolution of the mesh can be increased, thus allowing the evaluation of $\fe$ and hence estimates computation at trial mesh points that are closer to the current solution}. {\blc Note that unlike~\cite{AuHa2017}, the use of $\tau^{2}$ instead of $\tau$ has been motivated by the need to reduce the frame size parameter less aggressively during uncertain unsuccessful iterations as claimed next.} Indeed, in the case of uncertain unsuccessful iterations, i.e, when $\fsk -\fok $ belongs to the uncertainty interval $\mathcal{I}_{\gamma,\ef}(\dpl)$, the frame size parameter is reduced {\blc less aggressively}, specifically according to $\delta_p^{k+1}=\tau\dpl$, so that the uncertainty interval is reduced and as before, a new iteration is initiated with a new mesh size parameter $\delta_m^{k+1}$. {\blc An overview of the algorithm and its details are presented  in Figure~\ref{overview} and} Algorithm~\ref{algomads}.

\begin{figure}[]
	\centering
{\begin{tikzpicture} 
\draw [box] (-2,-1) rectangle (2,1);
\node[draw,ellipse,dashed,scale=0.8] (hidf) at (0,-0.23)  {\begin{tabular}{l}
	Hidden true \\	
	$\ $function $f$\\
	\end{tabular}};
\draw (0,0.4) node[above, scale=0.8]{Stochastic blackbox $f_\Theta$};
\node [draw, diamond, text width=4cm, text centered, aspect=2, scale=0.8, xscale = 1, fill = gray!15] (desacc) at (2.5,-5) {The desired accuracy of the estimate of $f(x)$ is reached}; 
\node[box, scale=0.8, text width=4cm, text centered] (ret1) at (-3,-5) {Re-evaluate $x$};
\node[draw,rectangle,rounded  corners=0pt, fill=gray!15, minimum  width=10cm, minimum  height=2cm, scale=0.8, text width=12cm, text centered] (pollstep) at (0,-8) {Update of the SEARCH or the POLL and the iterate $x$};
\node[draw,rectangle,rounded  corners=0pt, minimum  width=15cm, minimum  height=8.5cm, scale=0.8, text width=6cm, text centered] (rect2) at (0,-6) {};
\draw (-3.8,-3.5) node[above, scale=1]{\textbf{StoMADS}};
\draw  (2,0)  -|  (7,-5) node[pos = 0.25, above, scale=0.8]{$f_\Theta(x)$};
\draw[->,>=latex]  (7,-5)  --  (desacc);
\draw[->,>=latex]  (desacc)  --  (2.5,-7.21) node[pos = 0.5, right, scale=0.8]{yes};
\draw[->,>=latex]  (desacc)  --  (ret1) node[pos = 0.5, above, scale=0.8]{no};
\draw[->,>=latex]  (ret1)  -- (-6.5,-5) ;
\draw  (-6.5,-5)  -|  (-7,0);
\draw[->,>=latex]  (-7,0)  --  (-2,0) node[pos = 0.5, above, scale=0.8]{Current iterate $x$};
\draw[->,>=latex]  (pollstep)  -| (-7,-6) ;
\draw  (-7,-6.2)  -- (-7,-5) ;
\end{tikzpicture}}
\centering	
\caption[]{\small{Overview of the StoMADS algorithm. Given a current iterate $x$, an estimate of a desired accuracy of $f(x)$ is computed using the blackbox $f_\Theta$ evaluations. {\blc Such estimate is used} during the SEARCH or the POLL to check success and failure. {\blc Then, $x$ is updated unless the current iteration is unsuccessful} and a new iteration is initiated.}}
\label{overview}
\end{figure}
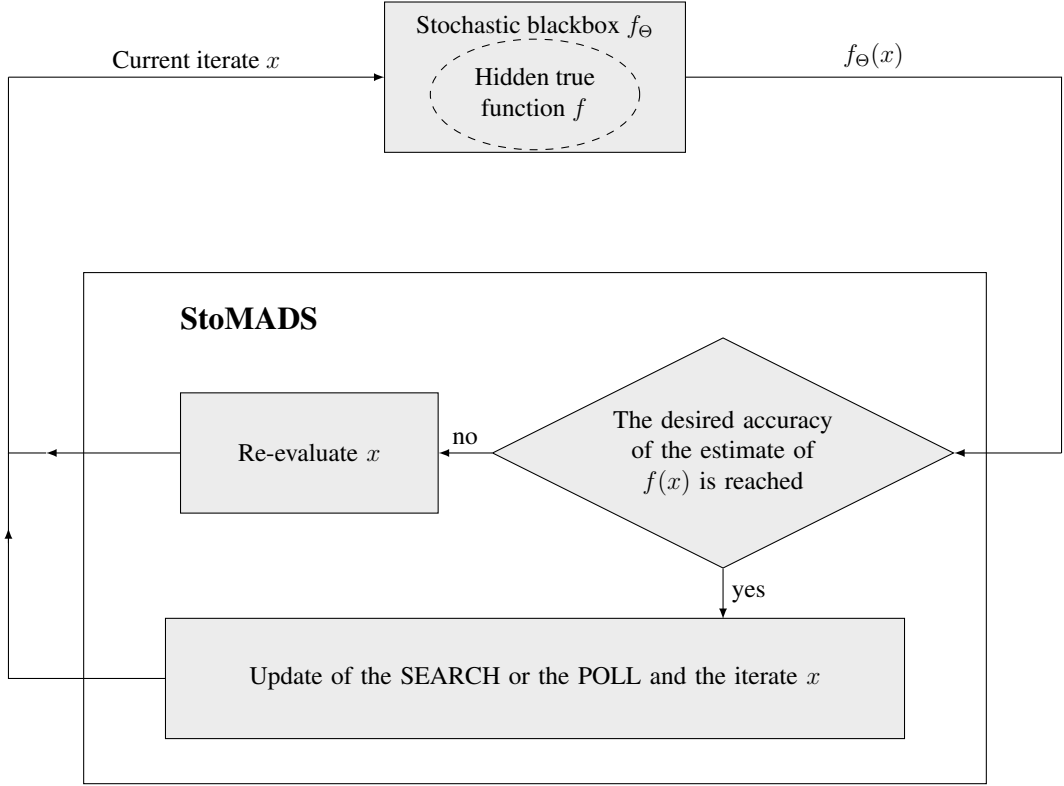
\begin{figure*}[ht!]
\begin{algorithm}[H]
	\caption{StoMADS}
	\label{algomads}
	\textbf{[0] Initialization}\\
	\hspace*{10mm}Choose $x^0\in\rn$, $\delta^0_p=1$, ${\blc \tau=\frac{1}{2}}$, $\ef>0$, $\epsilon_{stop}\geq 0$, $\gamma>2$ and ${\blc \hat{z}\in\N^*}$.\\
	\hspace*{10mm}Set the iteration counter $k \gets 0$.\\
	\textbf{[1] Parameter Update}\\
	\hspace*{10mm}Set the mesh size parameter to $\delta^k_m\gets\min \{\dpl,(\dpl)^2\}$.\\   
	{\blc \textbf{[2] Search}\\
	\hspace*{10mm}Select a finite subset $\mathcal{S}^k$ of $\mathcal{M}^k$.\\
	\hspace*{10mm}Obtain estimates $\fok$ and $\fsk$ of $\fd$ respectively at $x^k$ and $x^k+s^k\in \mathcal{S}^k$, using blackbox\\
    \hspace*{10mm}{\blc evaluations.}\\	
    \hspace*{10mm}If $\fsk -\fok \leq -\gamma\ef(\dpl)^2$ for some $x^k+s^k\in \mathcal{S}^k$,\\
    \hspace*{18mm}{\blc set} $x^{k+1}\gets x^k+s^k$ and $\delta_p^{k+1}\gets{\blc \min\{\tau^{-2}\dpl,\dmax\}}$ and go to \textbf{[4]}. \\
    \hspace*{10mm}Go to \textbf{[3]}.}\\
	\textbf{[3] Poll}\\
	\hspace*{10mm}Select a positive spanning set $\mathbb{D}^k_p$ such that $x^k+\dm d \in \mathcal{F}^k$ for all $d\in \mathbb{D}^k_p$.\\  
	\hspace*{10mm}Obtain estimates $\fok$ and $\fsk$ of $\fd(x^k)$ and $\fd(x^k+s^k)$, 
	respectively, {\blc using blackbox}\\
	\hspace*{10mm}{\blc evaluations.}\\
	\hspace*{10mm}\textbf{{\blc Success}}\\
	\hspace*{10mm}If $\fsk -\fok \leq -\gamma\ef(\dpl)^2$ for some $s^k=\dm d^k\in \{\dm d: d\in \mathbb{D}^k_p\}$,  \\
	\hspace*{10mm}{\blc set} $x^{k+1}\gets x^k+s^k$, and $\delta_p^{k+1}\gets{\blc \min\{\tau^{-2}\dpl,\dmax\}}$. \\
	\hspace*{10mm}\textbf{{\blc Failure}}\\
	\hspace*{18mm}\textbf{{\blc Certain failure:}} Otherwise if $\fsk -\fok \geq \gamma\ef(\dpl)^2$ for all $s^k\in \{\dm d: d\in \mathbb{D}^k_p\}$, \\
	\hspace*{18mm}set $x^{k+1}\gets x^k$ and $\delta^{k+1}_p\gets\tau^2\dpl$. \\
	\hspace*{18mm}\textbf{{\blc Uncertain failure:}} Otherwise, set $x^{k+1}\gets x^k$ and $\delta^{k+1}_p\gets\tau\dpl$.\\
	\textbf{[4] Termination}\\
	\hspace*{10mm}If $\dpl\geq\epsilon_{stop}$, \\
	\hspace*{10mm}set $k\gets k+1$ and go to \textbf{[1]}.\\
	\hspace*{10mm}Otherwise stop.
\end{algorithm}
\caption[]{\small{Pseudo code of the StoMADS algorithm. At iteration $k$, during the SEARCH or POLL, success {\blu or failure is determined} using {\blc information provided by both estimates} $\fok$ and $\fsk$ {\blc in order to update} the current iterate $x^k$ and the frame size parameter $\dpl$. A new iteration is then initiated with a new mesh size parameter $\delta^{k+1}_m$ unless a stopping criterion is met.}}
\end{figure*}


\subsection{Probabilistic estimates}
All the random variables in this work are defined on the same probability  space $(\Omega,{\blc \mathcal{G}}, \mathbb{P})$, $\Omega$ being the sample space, ${\blc \mathcal{G}}$ being a $\sigma$-algebra, that is a collection of all the events (subsets of $\Omega$) and $\mathbb{P}$ is a probability measure, that is a function that returns an event's probability. Any single outcome from the sample space $\Omega$ will be denoted by $\omega$. In general, random variables will be denoted by uppercase letters within the proposed algorithmic framework, while their realizations will be denoted by lowercase letters.
 
The estimates $\fok$ and $\fsk$ of function values are constructed at each iteration of Algorithm~\ref{algomads}, using {\blc evaluations} of the noisy blackbox $f_\Theta$. Because of the randomness of $f_\Theta$, such estimates can be respectively considered as realizations of random estimates {\blc denoted by} $\Fok$ and $\Fsk$, obtained based on some random samples of the stochastic function $f_\Theta(x)$. The behavior of $\Fok$ and $\Fsk$ then influences each iteration of Algorithm~\ref{algomads} (as it is the case in~\cite{chen2018stochastic,paquette2018stochastic,wang2019stochastic}) in such a way that the iterates {\blu $X^k$}, the polling directions {\blu $D^k$}, the mesh size parameter {\blu $\Dm$} and the frame size parameter {\blu $\Dp$} are also random quantities. {\blu $d^k=D^k(\omega)$, $x_k=X^k(\omega)$, $\dpl=\Dp(\omega)$ and $\dm=\Dm(\omega)$ denote respectively realizations of the random variables $D^k,X^k,\Dp$ and $\Dm$. Similarly, $S^k$ denotes the random variable with realizations $s^k$;  $\fok=\Fok(\omega)$ and $\fsk=\Fsk(\omega)$, where $\{\Fok,\Fsk\}$ denote estimates of $f(X^k)$ and $f(X^k+S^k)$ respectively. In other words, Algorithm~\ref{algomads} results in a stochastic process $\{X^k,S^k,\Dp,\Dm,\Fok,\Fsk\}$.} {\blc However, note that since $X^k$ is a random variable and not a vector of $\rn$, the notation ``$f(X^k)$'' is used to denote the random variable with realizations $f(X^k(\omega))$.}

The goal of this work is to show that the resulting stochastic process converges with probability one under some assumptions on  $\{\Fok,\Fsk\}$. In particular, such estimates will be assumed to be accurate with a sufficiently large but fixed probability, {\blc \it ``conditioned on the past''}~\cite{blanchet2016convergence,chen2018stochastic}.  

The notion of conditioning on the past is formalized as follows as proposed in~\cite{chen2018stochastic,paquette2018stochastic}. Let $\mathcal{F}^F_{k-1}$ denote the $\sigma$-algebra generated by $F^0_0, F^0_s,F^1_0, F^1_s,\dots, F^{k-1}_0$ and $F^{k-1}_s$. For completeness, $\mathcal{F}^F_{-1}$ {\blc is set to equal} $\sigma(x^0)$. Thus, $\{\mathcal{F}^F_{k}\}_{k\geq -1}$ is a filtration, that is a subsequence of increasing $\sigma$-algebras of~${\blc \mathcal{G}}$. Closeness or sufficient accuracy of function estimates is measured using the current frame size parameter. This notion is formalized, using the following definition which is a modified version of those in~\cite{blanchet2016convergence,cartis2018,chen2018stochastic,paquette2018stochastic} and which is similar to that in~\cite{wang2019stochastic}.

\begin{definition}\label{probestim2}
A sequence of random estimates $\{F^k_0,F^k_s\}$ is said to be $\beta$-probabilistically $\ef$-accurate with respect to the corresponding sequence $\{X^k, S^k, \Dp\}$ if the events 
\begin{equation}\label{event1}
J_k=\{F^k_0, F^k_s,\ \text{are}\ \ef\text{-accurate estimates of}\ \fd(x^k)\ \text{and}\ \fd(x^k+s^k),\ \text{respectively}\}\nonumber
\end{equation}
satisfy the following submartingale-like condition
\begin{equation}\label{beta1}
{\blc \pr{J_k\ |\ \mathcal{F}^F_{k-1}}}={\blc \E{\ijk\ |\ \mathcal{F}^F_{k-1}}}\geq \beta,\nonumber
\end{equation}
where $\ijk$ denotes the indicator function of the event $J_k$, that is $\ijk=1$ if $\omega\in J_k$ and $0$ otherwise.

{\blc An iteration $k$ is called ``{\it true}'' and an estimate is called ``{\it good}'' if $\ijk=1$. Otherwise the iteration is called ``{\it false}'' and the estimate is called ``{\it bad}''.}
\end{definition}

The following definition of $p$-integrable random variables~\cite{bhattacharya2007basic} is useful for the analysis of Algorithm~\ref{algomads}. 

\begin{definition}\label{pIntegrableDef}
Let $p\in [1,+\infty[$ be an integer and $(\Omega,{\blc \mathcal{G}},\mathbb{P})$ be a probability space. Then the Space $\lp$ of so-called $p$-integrable random variables is the set of all real-valued random variables $X$ such that
\begin{equation}\label{pIntegrableIneq}
\normp{X}:=\left(\int_{\Omega} \abs{X(\omega)}^p\pr{d\omega} \right)^{\frac{1}{p}}=\left(\E{\abs{X}^p} \right)^{\frac{1}{p}}<+\infty.\nonumber
\end{equation}

\end{definition}

{ In order for the random variable $f(X^k)$ to be integrable so that the conditional expectation {\blc $\E{f(X^k)\ |\ \mathcal{F}_{k-1}^F}$} can be well defined~\cite{bhattacharya2007basic} for the needs of the analysis of StoMADS, the following is assumed.

\begin{assumption}\label{assumpOniterates} {\blc The} objective {\blc function} $f$ is locally $\tilde{L}$-Lipschitz continuous {\blc everywhere and all} iterates $x^k$ generated by Algorithm~\ref{algomads} lie in a compact set $\mathcal{X}$.
\end{assumption}

Note that a similar assumption was made in the stochastic framework of~\cite{blanchet2016convergence} in order to ensure that there exists an upper bound $F_{\max}$ satisfying $f(x)\leq F_{\max}$ for all $x$ in a bounded and open set {\blc containing} all the iterates $x^k$ of the analyzed algorithm.

The following result shows that $f(X^k)$ is integrable if Assumption~\ref{assumpOniterates} holds.
\begin{proposition}\label{fIntegr}
{\blc If} Assumption~\ref{assumpOniterates} holds, then both $\Dp$ and  $f(X^k)\in\lun$ for all $k$.
\end{proposition}
\begin{proof}
{\blc The function $f$ is bounded on $\mathcal{X}$ since $f$ is locally Lipschitz and $\mathcal{X}$ is compact.} Consequently, there exists a finite constant $F_{\max}$ such that all the iterates $x^k$ lying in $\mathcal{X}$ satisfy $\abs{f(x^k)}\leq F_{\max}$. In other words, all realizations $f(X^k(\omega))$ of the random variable $f(X^k)$ satisfy $\abs{f(X^k(\omega))}\leq F_{\max}$. Therefore, $\E{\abs{f(X^k)}}:=\int_{\Omega}\abs{f(X^k(\omega))} \mathbb{P}(d\omega)\leq F_{\max}<+\infty$.

However, the integrability of $\Dp$ and hence, that of $\Dm$ follows straightforwardly from the fact that for all $\omega\in\Omega$, $\Dp(\omega)\leq \dmax$. Indeed, $\E{\abs{\Dp}}:=\int_{\Omega}\abs{\Dp(\omega)} \mathbb{P}(d\omega)\leq \dmax<+\infty$.
\end{proof}
}

The following key assumption similar to that made in~\cite{paquette2018stochastic} on the nature of the stochastic information in Algorithm~\ref{algomads} will be useful for the convergence analysis presented in Section~\ref{convAnalysis}.

\begin{assumption}\label{assump24} Let $\ef>0$ be {\blc the constant} of Proposition~\ref{decrease1}. The following holds for the {\blc random} quantities {\blc derived from} the algorithm:
\begin{itemize}
\item[(i)] The sequence of estimates $\{F^k_0,F^k_s\}$ generated by Algorithm~\ref{algomads} is $\beta$-probabilistically $\ef$-accurate for some $\beta\in(0,1)$.
\item[(ii)]There exists $\kappa_F>0$ such that the sequence of estimates $\{F^k_0, F^k_s\}$ generated by Algorithm~\ref{algomads} satisfies the following $\kappa_F$-variance condition for all $k\geq 0$,
\begin{eqnarray}\label{varcond1}
& &\E{\abs{\Fsk-\fd(X^k+S^k)}^2 |\ \mathcal{F}^F_{k-1}}\leq (\kappa_F)^2(\Dp)^4\nonumber\\
\text{and}\quad & &\E{\abs{\Fok-\fd(X^k)}^2 |\ \mathcal{F}^F_{k-1}}\leq (\kappa_F)^2(\Dp)^4.
\end{eqnarray}
\end{itemize}
\end{assumption}

\begin{remark} In regard to Assumption~\ref{assump24}, note that the role of the frame size parameter $\Dp$ in the stochastic framework of this work is twofold. First, it updates the resolution of the mesh (which, as it will be seen, gets infinitely fine) as mentioned earlier, and second, it adaptively controls the variance which again, as it will be seen, will be driven to zero when Algorithm~\ref{algomads} progresses, thus allowing it to reach a desired accuracy. Therefore, no other ``control size'' parameter {\blc is} required {\blc for} the analysis in order to control the variance as needed and described for the line search method proposed in~\cite{paquette2018stochastic}.
As in~\cite{paquette2018stochastic}, note that at point (ii) of Assumption~\ref{assump24}, the integrability of random quantities $\abs{\Fok-\fd(X^k)}^2$ and $\abs{\Fsk-\fd(X^k+S^k)}^2$ and hence straightforwardly that of $\abs{\Fok-\fd(X^k)}$ and $\abs{\Fsk-\fd(X^k+S^k)}$ is implicitly assumed for all~$k$.
\end{remark}

Using this key assumption on the accuracy of function estimates, a lower bound on $\beta$, defined in term of $\tau$, $\kappa_F$ and $\ef$ will be derived, under which convergence of Algorithm~\ref{algomads} holds. Before delving into the convergence analysis at Section~\ref{convAnalysis},  next {\blc is stated and proved} a useful lemma {\blc slightly modified from}~\cite{paquette2018stochastic}, showing the relationship between the variance assumption on the function values and the probability of obtaining bad estimates.

\begin{lemma}\label{lemma25}
Let Assumption~\ref{assump24} hold. Suppose $\{X^k, F^k_0, F^k_s, \Dp\}$ is a random process generated by Algorithm~\ref{algomads}. Then for every $k\geq 0$, 
\begin{eqnarray}\label{varcond2}
& &\E{\ijkc\abs{\Fsk-\fd(X^k+S^k)}\ |\ \mathcal{F}^F_{k-1}}\leq (1-\beta)^{1/2}\kappa_F(\Dp)^2\nonumber\\
\text{and}\quad & &\E{\ijkc\abs{\Fok-\fd(X^k)}\ |\ \mathcal{F}^F_{k-1}}\leq (1-\beta)^{1/2}\kappa_F(\Dp)^2.\nonumber
\end{eqnarray}
\end{lemma}

\begin{proof}
The result is shown for $\Fok-\fd(X^k)$ using ideas derived from~\cite{paquette2018stochastic}, specifically by making use of the conditional Cauchy-Schwarz inequality~\cite{bhattacharya2007basic}, but the proof for $\Fsk-\fd(X^k+S^k)$ is the same. However, the proof here is slightly modified compared to that in~\cite{paquette2018stochastic} in order to emphasize the integrability of the random variables that define the conditional expectations. 

Since it follows from Assumption~\ref{assump24} that  $\abs{\Fok-\fd(X^k)}\in\ldeux$ and that $\ijkc\in\ldeux$ trivially, then $\ijkc\abs{\Fok-\fd(X^k)}\in\lun$ thanks to the Cauchy-Schwarz inequality~\cite{bhattacharya2007basic}. Thus, it follows from the conditional Cauchy-Schwarz inequality that 
\begin{eqnarray}
\E{ \ijkc\abs{\Fok-\fd(X^k)}\ |\  \mathcal{F}^F_{k-1}} &\leq& \left[\E{\ijkc\ |\ \mathcal{F}^F_{k-1}}\right]^{1/2} \left[\E{ \abs{\Fok-\fd(X^k)}^2 |\   \mathcal{F}^F_{k-1}}\right]^{1/2}\nonumber \\
&\leq& (1-\beta)^{1/2}\kappa_F(\Dp)^2,\nonumber
\end{eqnarray}
where the last inequality follows from~\eqref{varcond1} and the fact that $\E{\ijkc\ |\ \mathcal{F}^F_{k-1}}=\pr{\bar{J_k}\ |\ \mathcal{F}^F_{k-1}}\leq 1-\beta$ thanks to the point (i) of Assumption~\ref{assump24}.
\end{proof}

\subsection{{\blr Computation of probabilistic estimates}}\label{computation}

This section {\blu demonstrates} how random estimates $\Fok$ and $\Fsk$ satisfying Assumption~\ref{assump24} can be constructed in a simple {\blu random} noise framework {\blc and hence how deterministic estimates $\fok$ and $\fsk$ can be obtained using evaluations of the blackbox $f_\Theta$}. However, note that since full details about such estimates construction are already provided in~\cite{chen2018stochastic,paquette2018stochastic,wang2019stochastic}, they are not provided here again.   

Now recall that $f_\Theta$ denote{\blr s} the noisy available blackbox which is the  computable version of the numerically unavailable objective $f$ and consider the following typical noise assumption often used in stochastic optimization literature~\cite{chen2018stochastic}, i.e, suppose that the noise $\Theta$ is unbiased for all $f$, that is, 
\begin{eqnarray*}
\Esp_\Theta[f_\Theta(x)]&=&f(x), \! \quad\quad\quad{\blr \text{for all}\ x,}\\
\text{and}\quad
\text{Var}_\Theta[f_\Theta(x)]&\leq& V<+\infty, \quad{\blr \text{for all}\ x,}
\end{eqnarray*}
where $V>0$ is a constant. Let $\Theta_1$, $\Theta_2$, {\blr $\Theta_3$ and $\Theta_4$ be four} independent random variables following the same distribution as $\Theta$. Define estimates $\Fok$ and $\Fsk$ respectively by $\Fok = \frac{1}{p^k}\sum_{i=1}^{p^k}f_{\Theta_{1,i}}(x^k)$ and $\Fsk = \frac{1}{p^k}\sum_{i=1}^{p^k}f_{\Theta_{2,i}}(x^k+s^k)$, where {\blr $p^k$ denotes the sample size,}  $\Theta_{1,1}, \Theta_{1,2}, \dots,\Theta_{1,p^k}$ and $\Theta_{2,1}, \Theta_{2,2}, \dots,\Theta_{2,p^k}$ are independent random samples of $\Theta_1$ and $\Theta_2$ respectively.
Thus, the random estimates $\Fok$ and $\Fsk$ satisfy Assumption~\ref{assump24}, provided that $p^k\geq \frac{V}{(\ef)^2(\dpl)^4(1-\sqrt{\beta})}.$ {\blc By using the fact that the deterministic estimates $\fok$ and $\fsk$ are realizations of $\Fok$ and $\Fsk$, respectively, it is then obvious to notice that their respective values can be obtained by averaging $p^k$ realizations of $f_\Theta$, resulting from the evaluations of the stochastic blackbox, respectively at $x^k$ and $x^k+s^k$.} 

{\blr Finally, the following technique using fewer blackbox evaluations is proposed especially for blackboxes that are expensive in term of evaluations. First, recall that $x^{k+1}=x^k+s^k$ and $x^{k+1}=x^k$ respectively on successful and unsuccessful iterations and denote by $n^k\leq p^k$ the number of blackbox evaluations at a given point when constructing an estimate at the iteration $k$, with {\blr $n^0=p^0$}. Let $\Theta_{3,1}, \Theta_{3,2}, \dots,\Theta_{3,n^{k+1}}$ and  $\Theta_{4,1}, \Theta_{4,2}, \dots,\Theta_{4,n^{k+1}}$ be independent random samples of $\Theta_3$ and $\Theta_4$ respectively. Then, when the iteration $k$ is successful, by noticing that $\fsk=\frac{1}{n^k}\sum_{i=1}^{n^k}f_{\theta_{2,i}}(x^k+s^k)$, the estimate $f_0^{k+1}$ of $f(x^{k+1})$ is computed according to
\begin{equation}\label{comp1}
f_0^{k+1}=\frac{n^k\fsk +\sum_{j=1}^{n^{k+1}}f_{\theta_{3,j}}(x^{k+1})}{p^{k+1}}  
\end{equation} 
where $p^{k+1}=n^k+n^{k+1}$, while after an unsuccessful iteration $k$, $f_0^{k+1}$ is given by
\begin{equation}\label{comp2}
f_0^{k+1}=\frac{p^k\fok +\sum_{j=1}^{n^{k+1}}f_{\theta_{4,j}}(x^{k+1})}{p^{k+1}}  
\end{equation} 
where $p^{k+1}=p^k+n^{k+1}$, $\theta_{3,j}$ and $\theta_{4,j}$,  $j\in\{1,2,\dots,n^{k+1}\}$, are the realizations, respectively, of the random variables $\Theta_{3,j}$ and $\Theta_{4,j}$. Indeed, this procedure used in Section~\ref{sec4}, improves the estimates accuracy by making use of available samples at the current iterate during estimates computation, thus avoiding additional blackbox evaluations and seems to be very useful for blackboxes that are expensive in term of evaluations.}

\section{Convergence analysis}\label{convAnalysis} 
{\blr This section presents convergence results of StoMADS using ideas inspired by~\cite{chen2018stochastic,LaBi2016,paquette2018stochastic}}. {\blc They are the stochastic variant of those of~\cite{AuDe2006} for MADS.} The first result is a {\it zero-order} result~\cite{AuDeLe07}, i.e, there exists a subsequence of {\blc the} StoMADS random iterates with realizations on meshes getting infinitely fine and which converges to a limit with probability one. 
More formally, StoMADS generates a convergent subsequence $\{X^k\}_{k\in K}$ of {\blc random} iterates such that ${\lim}_{k\in K}\ X^k=\hat{X} $ almost surely provided that {\blc ${\lim}_{k\to+\infty}\ \Dm=0 $ with probability one;}
a result {\blc which is} stronger than the liminf-type result of~\cite{AuDe2006} about the convergence of the sequence of mesh size parameters.
Then, under assumptions {\blc of} the compactness of the set containing all iterates and local Lipschitz {\blc continuity of} $f$, {\blc a stochastic variant of the {\it first-order} necessary optimality condition}~\cite{AuDe2006,AuHa2017} via the Clarke derivative~\cite{Clar83a} is proved.

\subsection{{\blu Zero}-order convergence}
In order to prove the existence of an almost surely convergent subsequence of StoMADS random iterates with realizations on meshes getting infinitely fine, it is first proved that with probability one, the sequence of random mesh size parameters converges to zero almost surely and then, there exists an almost surely convergent subsequence of StoMADS random iterates.

The following lemma similar to those derived in~\cite{chen2018stochastic,paquette2018stochastic}, guarantees an amount of decrease in the objective function $f$ when true successful iterations occur.

\begin{lemma}\label{lemma43}
Let $\ef>0$ and $\gamma>2$ be fixed constants and suppose $\{\fok,\fsk\}$ are $\ef$-accurate estimates. If the iteration is successful, then the improvement in $f$ is bounded as follows
\begin{equation}\label{improvement}
\fd(x^{k+1})-\fd(x^k)\leq  -  (\gamma-2)\ef(\dpl)^2.
\end{equation}
\end{lemma}

\begin{proof}
Since the iteration is successful and because the estimates are $\ef$-accurate,
\begin{eqnarray*}
\fd(x^k+s^k)-\fd(x^k) &=& \fd(x^k+s^k)-\fsk+(\fsk-\fok)+\fok-\fd(x^k)\\
&\leq& \ef(\dpl)^2 - \gamma\ef(\dpl)^2 + \ef(\dpl)^2 \\
&\leq& - (\gamma-2)\ef(\dpl)^2.
\end{eqnarray*}
\end{proof}

Before proving the following theorem that provides a result which is similar to that obtained in~\cite{chen2018stochastic} and which represents the corner stone of the convergence results in the present work, the following assumption on $f$ is needed.
\begin{assumption}\label{assumpOnf}
The function $f$ is 
bounded from below, i.e, there exists $f_{\min}\in\R$ such that $-\infty < f_{\min}\leq \fd(x)$,$\ $ for all $x\in\rn$. 
\end{assumption}
The following theorem states that the sequence of mesh size parameter $\{\Dm\}$ converges to zero with probability one.

\begin{theorem}\label{zerothOrder} 
Let Assumption~\ref{assumpOnf} be satisfied. Let $\ef>0$, $\tau\in(0,1)\cap\mathbb{Q}$ and $\gamma>2$. Let $\nu\in(0,1)$ be chosen such that 
\begin{equation}\label{nuchoice}
\frac{\nu}{1-\nu}\geq \frac{2(\tau^{-4}-1)}{\ef(\gamma-2)},
\end{equation}
and assume that Assumption~\ref{assump24} holds for $\beta\in(1/2,1)$ chosen such that 
\begin{equation}\label{betachoice}
\frac{\beta}{\sqrt{1-\beta}}\geq \frac{4\nu\kappa_F}{(1-\nu)(1-\tau^2)}.
\end{equation}
Then the sequence of mesh size parameter $\{\Dm\}$, generated by Algorithm~\ref{algomads} satisfies
\begin{equation}\label{almostsure}
\sum_{k=0}^{+\infty}\Dm < +\infty\quad\text{almost surely}.
\end{equation}
\end{theorem}

\begin{proof}

This theorem is proved, using techniques and ideas derived from~\cite{chen2018stochastic,LaBi2016,paquette2018stochastic} and making use of  properties of the following random function
\begin{equation}\label{randfunct}
\Phi_k=\nu(\fd(X^k)-f_{\min})+(1-\nu)(\Dp)^2,\nonumber
\end{equation}
a similar of which is used in~\cite{chen2018stochastic,LaBi2016}, where $\nu\in(0,1)$ is a fixed constant specified below. 
Recall that $\Dm=\min\{\Dp, (\Dp)^2\}$ {\blc and note that $\Phi_k\in\lun
$ according to Proposition~\ref{fIntegr}, which implies that the conditional expectation $\E{\Phi_{k+1}-\Phi_k|\mathcal{F}^F_{k-1}}$ is well defined for all $k$.}

The overall goal is to show that there exists a constant $\eta>0$ such that for all $k$, 
\begin{equation}\label{expect}
\E{\Phi_{k+1}-\Phi_k\ |\ \mathcal{F}^F_{k-1}}\leq -\eta(\Dp)^2<0.
\end{equation}
Indeed, assume~\eqref{expect} holds on every iteration. Since $f$ is bounded from below by $f_{\min}$ and $\Dp$ is positive, then $\Phi_k$ is bounded from below for all $k$. Hence, 
summing over $k\in\N$ and taking expectations on both sides of~\eqref{expect}, 
lead to the conclusion that~\eqref{almostsure} holds with probability~$1$. Thus, to prove the theorem, {\blc it is needed} to prove that on each iteration~\eqref{expect} holds. 

The proof of this theorem considers two separate cases: good estimates and bad estimates, each of which will be broken into whether an  iteration is successful, an unsuccessful iteration is certain or uncertain. For the sake of clarity of the analysis, let introduce the following events as suggested in~\cite{paquette2018stochastic}:\\
$S\ \ \ \! :=\{\text{The iteration is successful}\}$, $\quad\quad\quad\quad\quad \bar{S}\ \ \ \! :=\{\text{The iteration is unsuccessful}\}$,\\
${\bar{S}}^C:=\{\text{The unsuccessful iteration is certain}\}$, $\quad {\bar{S}}^{\bar{C}}:=\{\text{The unsuccessful iteration is uncertain}\}.$\\
\textbf{Case 1 (Good estimates, $\ijk = 1$).}
{\blc It will be shown} that $\Phi_k$ decreases no matter what type of iteration occurs and that the smallest decrease happens on the uncertain unsuccessful iteration. Thus, this case dominates the other two {\blc thus leading} overall {\blc to the conclusion that}
\begin{equation}\label{concl1}
\E{\ijk(\Phi_{k+1}-\Phi_k)\ |\ \mathcal{F}^F_{k-1}}\leq -\beta(1-\nu)(1-\tau^2)(\Dp)^2.
\end{equation}
\begin{itemize}
\item[(i)] \textit{Successful iteration} $(\isucc=1)$. The iteration is successful and estimates are good so a decrease in the objective $f$ occurs, specifically, lemma~\ref{lemma43} applies:
\begin{equation}
\ijk\isucc\ \nu(\fd(X^{k+1})-\fd(X^k))\leq -\ijk\isucc\nu(\gamma-2)\ef(\Dp)^2\label{aa1}
\end{equation}
As the iteration is successful, {\blc $\Delta^{k+1}_p= \min\{\tau^{-2}\Dp,\dmax\}$.} Consequently, 
\begin{equation}\label{aa2}
\ijk\isucc(1-\nu)\left[(\Delta^{k+1}_p)^2-(\Dp)^2\right]{\blc \leq}\ijk\isucc(1-\nu)(\tau^{-4}-1)(\Dp)^2.
\end{equation}
{\blc $\nu$ is chosen}  large enough so that the right-hand side term  of~\eqref{aa1} dominates that of~\eqref{aa2}, i.e, 
\begin{equation}
-\nu(\gamma-2)\ef(\Dp)^2+(1-\nu)(\tau^{-4}-1)(\Dp)^2 \leq -\frac{1}{2}\nu(\gamma-2)\ef(\Dp)^2,\label{A1}
\end{equation}
which is equivalent to equation~\eqref{nuchoice}. Then, the combination of~\eqref{aa1} and \eqref{aa2} leads to
\begin{eqnarray}\label{A3}
\ijk\isucc (\Phi_{k+1}-\Phi_k)\leq -\ijk\isucc \frac{1}{2}\nu(\gamma-2)\ef(\Dp)^2.
\end{eqnarray}
\item[(ii)] \textit{Certain unsuccessful iteration} $(\iunsuc=1)$. The iteration is unsuccessful, so there is a change of $0$ in the function values while $\Dp$ decreases. Hence,
\begin{eqnarray}
\ijk\isbar\iunsuc (\Phi_{k+1}-\Phi_k)&\leq& - \ijk\isbar\iunsuc(1-\nu)(1-\tau^4)(\Dp)^2\label{A5}
\end{eqnarray}
 
\item[(iii)] \textit{Uncertain unsuccessful iteration} $(\iunc=1)$. {\blc It is} easy to notice that the behavior of Algorithm~\ref{algomads} at {\it uncertain unsuccessful iteration} is obtained from that at {\it certain unsuccessful iteration} simply  by replacing $\tau^2$ by $\tau$. Thus, the bound in the change of $\Phi_k$ follows straightforwardly from~\eqref{A5} by replacing $\iunsuc$ by $\iunc$ and $\tau^4$ by $\tau^2$ as follows
\begin{eqnarray}\label{A8}
\ijk\isbar\iunc (\Phi_{k+1}-\Phi_k)\leq - \ijk\isbar\iunc(1-\nu)(1-\tau^2)(\Dp)^2.
\end{eqnarray}
{\blc $\nu$ is chosen} large enough so that uncertain unsuccessful iterations, specifically~\eqref{A8}, provide the worst case decrease when compared to~\eqref{A3} and~\eqref{A5}. More precisely, {\blc $\nu$ is chosen} according to 
\begin{eqnarray}
-\frac{1}{2}\nu(\gamma-2)\ef(\Dp)^2&\leq& -(1-\nu)(1-\tau^4)(\Dp)^2 \leq  -(1-\nu)(1-\tau^2)(\Dp)^2,\label{A6}
\end{eqnarray}
but using inequalities $1-\tau^2<1-\tau^4 < \tau^{-4}-1$, {\blc it can be noticed} that~\eqref{A6} is satisfied whenever $\nu$ is chosen according to~\eqref{A1}.

Thus, in the case of accurate estimates, using~\eqref{A3}, \eqref{A5}, \eqref{A8} and~\eqref{A6}, the change in $\Phi_k$ {\blc is bounded} by
\begin{eqnarray}\label{A9}
\ijk(\Phi_{k+1}-\Phi_k) &=& \ijk (\isucc+\isbar\iunsuc+\isbar\iunc)(\Phi_{k+1}-\Phi_k)\nonumber\\
&\leq& - \ijk(1-\nu)(1-\tau^2)(\Dp)^2.
\end{eqnarray}
Taking conditional expectations with respect to $\mathcal{F}^F_{k-1}$ in both sides of~\eqref{A9} and using assumption~\ref{assump24}, lead to~\eqref{concl1}.
\end{itemize}
\textbf{Case 2 (Bad estimates, $\ijkc = 1$).} Because of bad estimates, the algorithm can accept an iterate which leads to an increase in $f$ and $\Dp$, and hence in $\Phi_k$. To control this increase in $\Phi_k$, the variance in the function estimates {\blc is bounded making use of}~\eqref{varcond1}. Then, the probability of outcome (Case 2) {\blc is adjusted} to be sufficiently small in order to ensure that in expectation, $\Phi_k$ is sufficiently reduced. More precisely, {\blc it will be proved} that 
\begin{equation}\label{concl2}
\E{\ijkc(\Phi_{k+1}-\Phi_k)\ |\ \mathcal{F}^F_{k-1}}\leq 2\nu(1-\beta)^{1/2}\kappa_F(\Dp)^2.
\end{equation}
Whenever bad estimates occur, a successful iteration leads to the following bound
\begin{eqnarray}
\ijkc\isucc\ \nu(\fd(X^{k+1})-\fd(X^k)) &\leq& \ijkc\isucc\ \nu\left[(\Fsk -\Fok)+\abs{\fd(X^{k+1})-\Fsk}+\abs{\Fok-\fd(X^{k})}\right]\nonumber \\
&\leq& \ijkc\isucc\ \nu\left[-\gamma\ef(\Dp)^2+\abs{\fd(X^{k+1})-\Fsk}+\abs{\Fok-\fd(X^{k})}\right]\ \ \ \label{427}
\end{eqnarray}
where the last inequality is due to the decrease condition $\Fsk -\Fok\leq -\gamma\ef(\Dp)^2$ which holds at every successful iterations. As before, let consider three separate cases.
\begin{itemize}
\item[(i)] \textit{Successful iteration} $(\isucc=1)$. Since the iteration is successful, 
then as in Case~1, {\blc $\Delta^{k+1}_p= \min\{\tau^{-2}\Dp,\dmax\}$.}  Therefore,
\begin{equation}\label{aa22}
\ijkc\isucc(1-\nu)\left[(\Delta^{k+1}_p)^2-(\Dp)^2\right]{\leq}\ijkc\isucc(1-\nu)(\tau^{-4}-1)(\Dp)^2.
\end{equation}
By noticing that choosing $\nu$ according to~\eqref{A1} implies 
\begin{eqnarray}
-\nu \gamma\ef(\Dp)^2+(1-\nu)(\tau^{-4}-1)(\Dp)^2\leq 0,
\end{eqnarray}
then, combining~\eqref{427} and \eqref{aa22} leads to
\begin{eqnarray}\label{A10}
\ijkc\isucc (\Phi_{k+1}-\Phi_k)\leq \ijkc\isucc (\nu\abs{\fd(X^{k+1})-\Fsk}+\nu\abs{\Fok-\fd(X^{k})})
\end{eqnarray}
\item[(ii)] \textit{Certain unsuccessful iteration} $(\iunsuc=1)$. Since $\Dp$ is decreased and the change in function values is $0$, then the bound in the change of $\Phi_k$ follows straightforwardly from that obtained in~\eqref{A5} by replacing $\ijk$ by $\ijkc$. Specifically, 
\begin{eqnarray}
\ijkc\isbar\iunsuc (\Phi_{k+1}-\Phi_k)&\leq& - \ijkc\isbar\iunsuc(1-\nu)(1-\tau^4)(\Dp)^2\nonumber\\
&\leq&-\ijkc\isbar\iunsuc(1-\nu)(1-\tau^2)(\Dp)^2\label{A12}
\end{eqnarray}
\item[(iv)] \textit{Uncertain unsuccessful iteration} $(\iunc=1)$.
Here again, the bound in the change of $\Phi_k$ {\blc is derived from} that obtained in~\eqref{A8}, simply by replacing $\ijk$ by $\ijkc$. Specifically, 
\begin{eqnarray}\label{A13}
\ijkc\isbar\iunc (\Phi_{k+1}-\Phi_k)\leq - \ijkc\isbar\iunc(1-\nu)(1-\tau^2)(\Dp)^2.
\end{eqnarray}
By noticing that ${\bar{S}}^{\bar{C}}\cup {\bar{S}}^C = \bar{S}$, then combining~\eqref{A12} and~\eqref{A13} leads to
\begin{eqnarray}\label{A15}
\ijkc \isbar (\Phi_{k+1}-\Phi_k)\leq - \ijkc \isbar(1-\nu)(1-\tau^2)(\Dp)^2.
\end{eqnarray}
Finally, since~\eqref{A10} dominates~\eqref{A15}, then in all three cases, 
\begin{eqnarray}\label{A16}
\ijkc (\Phi_{k+1}-\Phi_k)\leq \ijkc(\nu\abs{\fd(X^{k+1})-\Fsk}+\nu\abs{\Fok-\fd(X^{k})}).
\end{eqnarray}
Taking expectation of~\eqref{A16} and applying lemma~\ref{lemma25} leads to~\eqref{concl2}.
\end{itemize}
Now, combining expectations~\eqref{concl1} and~\eqref{concl2} leads to
\begin{eqnarray}
\E{\Phi_{k+1}-\Phi_k\ |\ \mathcal{F}^{F}_{k-1}}&=& \E{(\ijk+\ijkc)(\Phi_{k+1}-\Phi_k)\ |\ \mathcal{F}^{F}_{k-1}}\nonumber \\
&\leq& -\beta(1-\nu)(1-\tau^2)(\Dp)^2+ 2 \nu (1-\beta)^{1/2}\kappa_F(\Dp)^2 \nonumber\\
&\leq& \left[-\beta(1-\nu)(1-\tau^2)+ 2 \nu\kappa_F (1-\beta)^{1/2}\right](\Dp)^2.\label{C1}
\end{eqnarray}
Then, choosing $\beta$ in $(1/2,1)$ according to~\eqref{betachoice} ensures that
\begin{equation}\label{C2}
-\beta(1-\nu)(1-\tau^2)+ 2 \nu\kappa_F (1-\beta)^{1/2}\leq -\frac{1}{2}\beta(1-\nu)(1-\tau^2).
\end{equation}
Hence, equation~\eqref{expect} follows from~\eqref{C1} and~\eqref{C2} with $\eta=\frac{1}{2}\beta(1-\nu)(1-\tau^2) > 0$, and the proof {\blc follows by} noticing that $\Dm = \min\{\Dp, (\Dp)^2\}$.
\end{proof}

The following result shows that with probability one, all realizations of random iterates $X^k$ generated by StoMADS lie on meshes getting infinitely fine.
\begin{corollary}\label{meshgetsfine}
Let the same assumptions that were made in Theorem~\ref{zerothOrder} hold. Then, almost surely, 
\begin{equation}\label{convmesh}
\underset{k\to+\infty}{\lim}\Dm=0.
\end{equation}
\end{corollary}
\begin{proof}
It follows from Theorem~\ref{zerothOrder} that $\sum_{k=0}^{+\infty} \Dm<+\infty$ almost surely. As a consequence, the sequence $\{\Dm\}_{k\in\N}$ of mesh size parameters converges to zero almost surely.
\end{proof}

\begin{remark}\label{limConv}
{\blc Let} emphasize that this latter result~\eqref{convmesh} {\blc is stronger} than the one obtained in the deterministic framework of the MADS algorithm where it has been proved that $\ {\liminf}_{k\to+\infty}\dm=0$. Indeed, unlike the deterministic framework of the MADS algorithm where available outputs of the objective function $f$ are directly compared in order to ensure improvement, such a behavior of the random sequence of mesh size parameters in the present stochastic framework is due to the use of a sufficient decrease condition in the definition of {\blc iteration types} (see Proposition~\ref{decrease1} and  Definition~\ref{decreasedef}). {\blc Note that a similar remark about the convergence to zero of a whole sequence of step size parameters is made in~\cite{CoScVibook} when a sufficient decrease condition had been imposed in the analyzed  ``Directional direct-search method''}.
\end{remark}
 
\begin{remark}\label{maxaccuracy}
Since the sequence $\{\Dp\}_{k\in\N}$ converges to zero almost surely according to Theorem~\ref{zerothOrder}, then both conditions of Assumption~\ref{assump24} (ii), 
that adaptively control  the variance in function estimates, drive the variance to zero, thus allowing Algorithm~\ref{algomads} to reach a desired accuracy where function estimates are representatives of their corresponding true function values.
\end{remark}

Next, in order to show the existence of convergent subsequences of StoMADS iterates, let introduce the following definition which is similar to that in~\cite{AuHa2017}. 

\begin{definition}\label{refiningsub}
A convergent subsequence $\{x^k\}_{k\in K}$ of the \emph{StoMADS} iterates 
(for some subset of indices $K$), is said to be a \emph{refining subsequence}, if and only if $\{\dm\}_{k\in K}$ converges to zero. The limit $\hat{x}$ of $\{x^k\}_{k\in K}$ is called a \emph{refined point}.
\end{definition}

The existence of convergent refining subsequences was proved by Audet and Dennis in the deterministic framework of the Generalized Pattern Search (GPS)~\cite{AuDe03a} algorithm under assumptions including that according to which all the iterates generated by GPS lie in a compact set. These authors then generalized the proof to the framework of the MADS algorithm in~\cite{AuDe2006}, but with the latter assumption replaced by that according to which all the iterates produced by MADS belong to the level set $\mathscr{L}(f(x^0)):=\{x\in\rn:f(x)\leq f(x^0)\}$ supposed to be bounded. For both algorithms, the refining subsequences was shown to be subsequences of mesh local optimizers on meshes getting infinitely fine. However, note that while in a deterministic framework, the objective values $f(x)$ can never increase from one iteration to another, the challenge as well of the analysis of StoMADS in the present stochastic framework as in those of related works~\cite{blanchet2016convergence,chen2018stochastic,paquette2018stochastic,wang2019stochastic} lies in the fact that the iterates produced can lie outside the initial level set $\mathscr{L}(f(x^0))$ since the objective values $f(x)$ can possibly increase easily between successive iterations. In other words, StoMADS ``{\it can venture outside the initial level set}''~\cite{chen2018stochastic}.  Thus, motivated by these latter remarks, the following theorem is proved under Assumption~\ref{assumpOniterates}, i.e, the same that was used in~\cite{AuDe03a}, in order to make the analysis simpler.

\begin{theorem}\label{refiningExist}
Let the assumptions that were made in Theorem~\ref{zerothOrder} and Assumption~\ref{assumpOniterates} hold. 
Then, there exists at least one almost surely convergent refining subsequence $\{X^k\}_{k\in K}$.
\end{theorem}
\begin{proof}
The proof uses ideas derived from~\cite{AuDe03a}. The result is proved by making use of the event $V=\{\omega\in\Omega: \lim_{k\to+\infty}\Dm(\omega)=0 \}$ that is almost sure thanks to Corollary~\ref{meshgetsfine}.

For all $\omega\in V$, $\{X^k(\omega)\}_{k\in\N}$ is a sequence of iterates on meshes getting infinitely fine. It therefore follows from the compactness hypothesis of Assumption~\ref{assumpOniterates} that there exists a subset of indices $K\subset \N$ for which the subsequence $\{X^k(\omega)\}_{k\in K}$ converges. Denote by $\hat{X}(\omega)$ the limit of $\{X^k(\omega)\}_{k\in K}$.
The proof {\blc follows by} noticing that $V\subseteq \{\omega\in\Omega: {\lim}_{k\in K}X^k(\omega)=\hat{X}(\omega)\}$.
\end{proof}

\subsection{{\blr Nonsmooth optimality conditions}}
The main goal of this subsection is to show with probability one {\blu that}, any refined point $\hat{X}$ derived in Theorem~\ref{refiningExist} satisfies a stochastic variant of the first-order necessary optimality condition via the Clarke derivative stated as Theorem 6.9 in~\cite{AuHa2017}. 

One of the most important requirements on which the Clarke optimality result relies is that the search directions $d^k$ should be chosen in such a way that the sequence $\{\dm\norminf{d^k}\}_{k\in\N}$ converges to zero while $\{\dpl\norminf{d^k}\}_{k\in\N}$ does not, even though both sequences of mesh and frame size parameters converge to zero. Thus, in order for such expectations to be met, {\blc the analysis in this subsection assumes that the columns of the matrix $\mathbf{D}$ used in the definition of the mesh $\mathcal{M}^k$ are the {\blr $2n$} positive and negative coordinate directions, the initial frame size parameter $\delta_p^0$ equals~$1$, the {\blr mesh refining} parameter $\tau$ equals~$1/2$ and moreover, all search directions used in Algorithm~\ref{algomads} during the POLL step are  generated by Algorithm~\ref{polldir} taken from~\cite{AuHa2017}. Note that under these previous assumptions, the sequence $\{\dm\norminf{d^k} \}_{k\in\N}$ is shown in~\cite{AuHa2017} to converge to zero. However, $\dpl\norminf{d^k}\geq 1$ for large values of $k$. Indeed, consider $d^k=\text{round}\left(\frac{\dpl}{\dm}\frac{h}{\norminf{h}}\right)$, where $h=(h^1,h^2,\dots,h^n)^{\top}$ is a column of the Householder matrix $\mathbf{H}^k$, an index $j$ such that $\abs{h^j}=\norminf{h}$ and $k_0$ such that $\dpl\leq 1$ for all $k \geq k_0$. Then, for all $k \geq k_0$, $\dpl\norminf{d^k}\geq 1$ since $1/\dpl$ is an integer and \[\dpl\ \text{round}\left(\abs{\frac{\dpl}{\dm}\frac{h^j}{\norminf{h}}} \right)=\dpl\ \text{round}\left(\frac{1}{\dpl} \right)=1.\]}
Note also that in Algorithm~\ref{polldir}, the Householder matrix is denoted by $\mathbf{H}^k$ instead of $H^k$~\cite{AuHa2017} so that it is not considered as a random matrix.

\begin{algorithm}[H] 
	\caption{Creating the set $\mathbb{D}^k_p$ of poll directions}
	\label{polldir}
	Given $v^k\in\rn$ with $\norme{v^k}=1$ and $\dpl\geq\dm >0$\\
	\textbf{[1] Create Householder matrix}\\
	\hspace*{10mm}Use $v^k$ to create its associated Householder matrix $\mathbf{H}^k=I-2v^k{v^k}^\top\in\rnn$ \\
	\hspace*{10mm}and let $\mathbf{H}^k=[h_1\ h_2\dots h_n]$ \\
	\textbf{[2] Create poll set}\\
	\hspace*{10mm}Define $\mathbb{B}^k=\{b_1,b_2,\dots,b_n\}$ with $b_j=\text{round}\left(\frac{\dpl}{\dm}\frac{h_j}{\norminf{h_j}}\right)\in \mathbb{Z}^n$\\
	\hspace*{10mm}set $\mathbb{D}^k_p=\mathbb{B}^k\cup (-\mathbb{B}^k)$
\end{algorithm}


The following auxiliary result~\cite{ BaScVi2014,chen2018stochastic} taken from martingale literature~\cite{durrett2010probability} will be useful later in the analysis.

\begin{theorem}\label{theorem44}
Let $\{G_k\}_{k\in\N}$ be a submartingale, i.e, a sequence of random variables which, for every $k\in\N$, satisfy 
\[\E{G_k|\mathcal{F}^G_{k-1}}\geq G_{k-1},\] 
where $\mathcal{F}^G_{k-1}=\sigma(G_0,G_1,\dots,G_{k-1})$ is the $\sigma$-algebra generated by $G_0,G_1,\dots,G_{k-1}$, and $\Esp(G_k|\mathcal{F}^G_{k-1})$ denotes the conditional expectation of $G_k$, given the past history of events $\mathcal{F}^G_{k-1}$.
	
Assume further that $G_k-G_{k-1}\leq M < +\infty$, for every $k$. Then, 
\begin{equation}\label{submartingale}
\pr{\left\lbrace \underset{k\to\infty}{\lim} G_k < \infty \right\rbrace \cup \left\lbrace \underset{k\to\infty}{\limsup}\ G_k = \infty \right\rbrace} = 1.\nonumber
\end{equation}
\end{theorem}

The properties of the random function ${\Psi}_k$ introduced next will be useful for the proof of the optimality result via the Clarke derivative in Theorem~\ref{clarkeOrder}.

\begin{theorem}\label{psi}
Let the same assumptions that were made in Theorem~\ref{zerothOrder} hold. Define the random function $\Psi_k$ with realizations $\psi_k$ as follows
\begin{equation}\label{psik1}
\psi_k=\frac{\fd(x^k)-\fd(x^k+\dm d)}{\dpl},\nonumber
\end{equation}
where $d\in\mathbb{D}^k_p$ is any direction used by \emph{StoMADS} and that is generated by Algorithm~\ref{polldir}. Then, almost surely,
\begin{eqnarray}\label{liminfPsik}
\underset{k\to+\infty}{\liminf}\ \Psi_k\leq 0.
\end{eqnarray}
\end{theorem}

\begin{proof}
Using ideas in the proof of the liminf-type first-order convergence result in~\cite{chen2018stochastic}, this result is proved by contradiction conditioned on the event $V'=\left\lbrace \lim_{k\to+\infty}\Dp=0 \right\rbrace$ that is almost sure thanks to Corollary~\ref{meshgetsfine}.
All that follows is conditioned on $V'$.
Assume that there exists $\epsilon >0$ such that, with positive probability, 
\begin{equation}\label{D3}
\Psi_k\geq \epsilon (\gamma+2), \quad\text{for all}\ k\in\N, 
\end{equation}
where $\gamma\in(2,+\infty)$ is the same constant in Algorithm~\ref{algomads} and recall that $s^k=\dm d$ for all $k$. Let $\{x^k\}_{k\in\N}$, $\{\dpl\}_{k\in\N}$ and $\{s^k\}_{k\in\N}$ be realizations of $\{X^k\}_{k\in\N}$, $\{\Dp\}_{k\in\N}$ and $\{S^k\}_{k\in\N}$, respectively for which $\psi_k\geq \epsilon (\gamma+2), \ \text{for all}\ k\in\N$. Since ${\lim}_{k\to +\infty}\dpl = 0$ because of the conditioning on $V'$, there exists $k_0\in\N$ such that 
\begin{eqnarray}\label{B11}
\dpl < \lambda:= \min \left\lbrace \frac{\epsilon}{\ef}, \tau^{2-\hat{z}} \right\rbrace,\ {\blc \text{for all  }\ k\geq k_0}.
\end{eqnarray}
Define the random variable $R_k$ with realizations $r_k=-\frac{1}{2}{\log}_\tau \left(\frac{\dpl}{\lambda} \right)$. Then, $r_k < 0$ for all $k\geq k_0$. The main idea of the proof is to show that such realizations occur only with probability zero, hence obtaining a contradiction. In order to first show that $R_k$ is a submartingale, recall the events $J_k$ in the Definition~\ref{probestim2} for some $\ef\in(0,1)$ and consider some iterate $k\geq k_0$ for which $J_k$ occurs, which happens with probability at least $\beta > 1/2$ thanks to Theorem~\ref{zerothOrder}. 
Now, noticing that~\eqref{D3} and~\eqref{B11} imply 
\begin{equation}\label{D4}
\fd(x^k+s^k)-\fd(x^k)\leq -\epsilon(\gamma+2)\dpl\leq -\ef (\gamma+2)(\dpl)^2,\ \text{for all}\ k\geq k_0,\nonumber
\end{equation}
then, for all $k\geq k_0$,
\begin{eqnarray*}
\fsk-\fok&=&[\fd(x^k+s^k)-\fd(x^k)]+[\fd(x^k)-\fok]+[\fsk-\fd(x^k+s^k)] \\
&\leq& -\ef (\gamma+2)(\dpl)^2+2\ef(\dpl)^2 = -\gamma\ef(\dpl)^2.
\end{eqnarray*}
Hence, the $k$-kth iteration of Algorithm~\ref{algomads} is successful, so the frame size parameter $\dpl$ is updated according to $\delta^{k+1}_p=\tau^{-2}\dpl$ since $\dpl<\tau^{2-\hat{z}}$. Consequently, $r_{k+1}=r_k+1$.
	
	
Let $\mathcal{F}^J_{k-1}=\sigma(J_0,J_1,\dots,J_{k-1})$. If $\ijk=0$, which occurs with probability at most $1-\beta$, then the inequality $\delta_p^{k+1}\geq \tau^2 \dpl$ always holds, which implies that $r_{k+1}\geq r_k-1$. Thus, 
\begin{eqnarray}\label{D5}
\E{\ijk(R_{k+1}-R_k)|\mathcal{F}^J_{k-1}}&=&\pr{J_k|\mathcal{F}^J_{k-1}}\geq\beta\nonumber \\
\text{and}\quad \E{\ijkc(R_{k+1}-R_k)|\mathcal{F}^J_{k-1}}&\geq& -\pr{\bar{J_k}|\mathcal{F}^J_{k-1}}\geq \beta-1\nonumber.
\end{eqnarray}
Hence, $\E{R_{k+1}-R_k|\mathcal{F}^J_{k-1}}\geq 2 \beta-1>0$, implying that $R_k$ is a submartingale. 
	
	
Now, construct the following random walk $W_k$ on the same probablity space as $R_k$, which will serve as a lower bound on $R_k$ and for which $\left\lbrace {\limsup}_{k\to+\infty} W_k = +\infty \right\rbrace$ holds almost surely
\begin{eqnarray*}
W_k=\sum_{i=0}^{k}(2 \cdot \iji - 1).
\end{eqnarray*}
From the submartingale-like property enforced in Definition~\ref{probestim2}, it easily follows that $W_k$ is a submartingale. In fact, 
\begin{eqnarray*}
\E{W_k|\mathcal{F}^J_{k-1}} &=& \E{W_{k-1}|\mathcal{F}^J_{k-1}} + \E{2 \cdot \ijk - 1|\mathcal{F}^J_{k-1}}\\
&=& W_{k-1}+2 \E{\ijk |\mathcal{F}^J_{k-1}}- 1\\
&=& W_{k-1}+2 \pr{J_k|\mathcal{F}^J_{k-1}}- 1\\
&\geq& W_{k-1}.
\end{eqnarray*}
Notice that the submartingale $W_k$ has $\pm 1$ and hence, bounded increments, whence cannot have a finite limit. Thus, it follows from Theorem~\ref{theorem44} that the event $\left\lbrace {\limsup}_{k\to+\infty}  W_k = +\infty \right\rbrace$ occurs almost surely. 

Since $R_k$ and $W_k$ are constructed in such a way that 
\begin{equation}
r_k-r_{k_0} = -\frac{1}{2}{\log}_\tau \left(\frac{\dpl}{\delta^{k_0}_p}\right) = k-k_0 \geq w_k-w_{k_0},\nonumber
\end{equation}
with $w_k$ denoting a realization of $W_k$, then with probability one, $R_k$ has to be positive infinitely often. 
Consequently, the sequence of realizations $r_k$ such that $r_k<0$ for all $k\geq k_0$ occurs with probability zero. Thus, the assumption that $\Psi_k\geq \epsilon (\gamma+2) \ \text{holds for all}\ k\in\N$ with positive probability is false and~\eqref{liminfPsik} holds almost surely.

\end{proof}

The following definition of refining directions~\cite{AuDe2006,AuHa2017} will be useful in the analysis.

\begin{definition}\label{refiningDir}
Given a convergent refining subsequence $\{x^k\}_{k\in K}$ and its corresponding refined point $\hat{x}$, a direction $d$ is said to be a \emph{refining direction} if and only if there exists an infinite subset $L\subseteq K$ with poll directions $d^k\in\mathbb{D}^k_p$ such that 
$\underset{k\in L}{\lim}\frac{d^k}{\norminf{d^k}}=\frac{d}{\norminf{d}}$.
\end{definition}
Note that for all realizations of StoMADS, the existence of a refining direction $d$ for a given refining subsequence $\{x^k\}_{k\in K}$ and its corresponding refined point $\hat{x}$ is justified by the compactness of the unit closed ball.

Next is stated a useful {\blc result taken from}~\cite{AuDe2006}, that provides {\blc in particular} a lower bound on the Clarke directional derivative.
\begin{lemma}\label{limsuplemma}
Let $f:\rn\to\R\ $ be {\blr locally} Lipschitz near $\hat{x}\in\rn$. Then the Clarke \emph{generalized directional derivative} of $f$ at $\hat{x}$ in the direction $d\in\rn$ satisfies
\begin{equation}\label{clarke}
f^{\circ}(\hat{x};d):= \underset{t\searrow 0}{\underset{y\to \hat{x}}{\limsup}} \frac{f(y+td)-f(y)}{t}{\blc = }\underset{x\to\hat{x}, v\to d, t\searrow 0}{\limsup} \frac{f(x+tv)-f(x)}{t}.\nonumber
\end{equation}
\end{lemma}

The following result proved using properties of the random function $\Psi_k$ defined in Theorem~\ref{psi} is a stochastic variant of that in~\cite{AuDe2006}. It states that with probability one, the Clarke generalized derivative of $f$ at a refined point in any corresponding refining direction is nonnegative. It is however worthwhile to mention that while the proof in~\cite{AuDe2006} relies on the fact that the inequality $\fd(x^k+\dm d^k)-\fd(x^k)\geq 0$ always holds on every unsuccessful iterations, the idea of proof in the present analysis is different since some of such unsuccessful iterations can be uncertain, in which case $\fd(x^k+\dm d^k)-\fd(x^k)$ belongs to the uncertainty interval $\mathcal{I}_{\gamma+2,\ef}(\delta^k_p)$. 
\begin{theorem}(Convergence of \emph{StoMADS}).\label{clarkeOrder}
Let the assumptions of Theorem~\ref{refiningExist} hold. Then, there exists an almost sure event $V''$ such that for all $\omega\in V''$, for all refined point $\hat{X}(\omega)\in\rn$ and for all refining directions $D(\omega)\in\rn$ for $\hat{X}(\omega)$, the generalized directional derivative of $f$ at $\hat{X}(\omega)$ in the direction $D(\omega)$ is nonnegative, i.e, 
\begin{eqnarray}
f^{\circ}\left(\hat{X}(\omega);D(\omega)\right)\geq 0.
\end{eqnarray}
 
\end{theorem}

\begin{proof}
It follows from Corollary~\ref{meshgetsfine} and Theorem~\ref{psi} that the event \[V'':=\left\lbrace \omega\in\Omega: \lim_{k\to+\infty}\Dm(\omega)=0 \right\rbrace \bigcap \left\lbrace \omega\in\Omega:\exists K'(\omega)\subset\N, \lim_{k\in K'(\omega)}\Psi_k(\omega)\leq 0 \right\rbrace \]
is almost sure as countable intersection of almost sure events. Consider some arbitrary outcome $\omega\in V''$. Denote $\tilde{K}=K'(\omega)$ and recall that $\dm=\Dm(\omega)$, $\dpl=\Dp(\omega)$  and $\psi_k=\Psi_k(\omega)$. Since $\lim_{k\in \tilde{K}}\dm=0$, then using arguments as in the proof of Theorem~\ref{refiningExist}, there exists a subset $K\subset\tilde{K}$ such that $\lim_{k\in K}x^k=\hat{x}$. It then follows from the compactness of the closed unit ball of $\rn$ that there exists a subset $L\subset K$ such that the normalized subsequence $\left\lbrace d^k/\norminf{d^k} \right\rbrace_{k\in L}$ of POLL directions used by StoMADS converges to a limit $d/\norminf{d}=D(\omega)/\norme{D(\omega)}_{\infty}$ and on the other hand, $\lim_{k\in L}\psi_k\leq 0$. 

Since $\dpl\norminf{d^k}$ does not approach $0$ even though $\lim_{k\in L}\dpl=0$, the following holds 
\begin{eqnarray}\label{D1}
\underset{k\in L}{\lim}\  \left(\frac{-\psi_k}{\dpl\norminf{d^k}} \right) = \underset{k\in L}{\lim}\ \frac{\fd(x^k+\dm d^k)-\fd(x^k)}{\dm\norminf{d^k}} \geq 0.
\end{eqnarray}
Then, applying Lemmas~\ref{limsuplemma} using sequences $x^k\to \hat{x}$, $d^k/{\norminf{d^k}}\to d/\norminf{d}$ and $\dm\norminf{d^k}\searrow~0$, the following holds for the generalized derivative of $f$:
\begin{eqnarray}
f^{\circ}\left(\hat{X}(\omega);\frac{D(\omega)}{\norme{D(\omega)}_{\infty}}\right)=f^{\circ}\left(\hat{x};\frac{d}{\norminf{d}}\right)&{\blc =}& \underset{x\to\hat{x}, v\to {d/\norminf{d}}, t\searrow 0}{\limsup} \frac{\fd(x+tv)-\fd(x)}{t}\nonumber \\
&\geq& \underset{k\in L}{\limsup}\ \frac{\fd\left(x^k+\dm {\norminf{d^k}} \frac{d^k}{\norminf{d^k}}\right)-\fd(x^k)}{\dm\norminf{d^k}}\nonumber\\
&\geq& \underset{k\in L}{\lim}\ \frac{\fd\left(x^k+\dm {\norminf{d^k}} \frac{d^k}{\norminf{d^k}}\right)-\fd(x^k)}{\dm\norminf{d^k}} \geq 0,\quad \quad  \label{D2}
\end{eqnarray}
where the last inequality in~\eqref{D2} follows from~\eqref{D1}. 

\end{proof}

\section{Computational study}\label{sec4}

The performance of StoMADS is analyzed in this section on a collection of stochastic noisy functions artificially created from deterministically {\blc unconstrained} analytical problems from the optimization literature. Several variants of StoMADS have been compared to Robust-MADS~\cite{AudIhaLedTrib2016} which is {\blc the current} noisy blackbox optimization algorithm available in the {\blc \sf  NOMAD}~\cite{Le09b} software package (version~3.9.1) {\blc  and which is referred to in this section as {\sf NOMAD-robust}. All tests with both StoMADS and {\sf NOMAD-robust} use only a POLL step, i.e, the SEARCH step and hence the quadratic models~\cite{CoLed2011} in {\sf NOMAD} are disabled, with the OrthoMADS $2n$ directions~\cite{AbAuDeLe09} ordered by means of an opportunistic strategy~\cite{AuHa2017} and disabling the anisotropic mesh~\blc \cite{AuLedTr2014}. The MADS algorithm~\cite{AuDe2006} with the SEARCH step disabled is referred to as {\sf NOMAD-basic}. The default algorithm in {\sf NOMAD} is referred to as {\sf NOMAD-default}. Note that detailed descriptions of all these algorithms are provided in Table~\ref{tabAlgo}. Moreover, in order to highlight the ability of StoMADS vis-‡-vis of {\sf NOMAD-basic} and {\sf NOMAD-default}, to cope with stochastically noisy optimization problems,
both latter algorithms are also compared to StoMADS.}

\begin{table}[ht!]
\footnotesize
\caption{Description of the algorithms.}
\label{tabAlgo} 
\centering
\begin{tabular}{llllccclllll}
	\hline
	\multirow{4}{*}{Algorithm} & \multicolumn{11}{c}{Description}                                                                                                                                                                                                                      \\ \cline{2-12} 
	& \multicolumn{3}{l}{Direction type}  & \begin{tabular}[c]{@{}c@{}}Anisotropic\\ mesh \end{tabular} & \begin{tabular}[c]{@{}c@{}}$\ \ $Opportunistic$\ \ $\\ strategy\end{tabular} & \multicolumn{6}{c}{\begin{tabular}[c]{@{}c@{}}Quadratic\\ models\end{tabular}} \\ \hline
	StoMADS                    & \multicolumn{3}{l}{OrthoMADS $2n$}  & No                                                          & Yes                                                              & \multicolumn{6}{c}{No}                                                         \\ \hline
	{\sf NOMAD-robust $\quad \quad$}               & \multicolumn{3}{l}{OrthoMADS $2n$}  & No                                                          & Yes                                                              & \multicolumn{6}{c}{No}                                                         \\ \hline
	{\sf NOMAD-basic}                & \multicolumn{3}{l}{OrthoMADS $2n$}  & No                                                          & Yes                                                              & \multicolumn{6}{c}{No}                                                         \\ \hline
	{\sf NOMAD-default}              & \multicolumn{3}{l}{OrthoMADS $n+1$~\cite{AuIaLeDTr2014}} & Yes                                                         & Yes                                                              & \multicolumn{6}{c}{Yes}                                                        \\ \hline
\end{tabular}
\end{table}


The analytical unconstrained problems are adapted from the $22$ different {\sf CUTEst}~\cite{cutest} functions used in~\cite{MoWi2009} with different starting points {\blr for} a total of $66$ unconstrained {\blc instances} whose dimensions range from $2$ to $12$. 
Their objectives are in the form of a sum of squares function, i.e, 
\begin{equation}\label{square}
f(x)=\sum_{i=1}^{m}(f_i(x))^2, \nonumber
\end{equation}
$f_i(x)$ being a smooth function for each $i\in\{1,2,\dots,m\}$. 

The type of noise that is tested is referred to as ``{\it additive}'' noise, i.e, each $f_i$ is additively perturbed by some random variable $\Theta_i$ generated uniformly in the interval $I(\sigma,x^0,f^*)$ defined by\\ ${\blu I(\sigma,x^0,f^*)=} \left[-\sigma\abs{f(x^0)-f^*},\sigma\abs{f(x^0)-f^*}\right]$, {\blu i.e.}, 
\begin{equation}\label{squaresto}
f_\Theta(x)=\sum_{i=1}^{m}(f_i(x)+\Theta_i)^2, 
\end{equation}
where $\sigma>0$ is a constant that is used to define different noise levels in the blackbox $f_\Theta$, $\Theta_i$, $i\in\{1,2,\dots,m\}$, are independent random variables, $x^0$ is a starting point and $f^*$ is the best known minimum value of $f$. Although it obviously follows from~\eqref{squaresto} that $\Esp_\Theta[f_\Theta(x)]=f(x)+\sum_{i=1}^{m}\Esp[(\Theta_i)^2]$, optimization results are not affected by this constant bias term since $\min_x \Esp_\Theta[f_\Theta(x)] = \min_x f(x)$. 

The {\sf NOMAD-robust} algorithm to which StoMADS is compared is a smoothing-based algorithm designed to {\blu handle} noisy blackbox optimization problems.  At each iteration of {\sf NOMAD-robust}, a best mesh local optimizer is determined based on values of the smoothed version of the noisy available objective constructed from a list of trial points and making use of a Gaussian kernel~\cite{AudIhaLedTrib2016}. This list is then updated with the best iterate found before the next iteration of the algorithm. Although experiments in~\cite{AudIhaLedTrib2016} have been conducted on deterministically noisy problems, the smoothing-based technique does not depend on the link between the objective function $f$ and its noisy available version, which means that {\sf NOMAD-robust} is supposed to cope with stochastically noisy problems. 

{\blc In order to assess if the algorithms have successfully generated solution values close to the best function $f$ values, data profiles~\cite{MoWi2009} and performance profiles~\cite{DoMo02,MoWi2009} are presented using the following convergence test: 
\begin{equation}\label{fNacc}
f(x^N)\leq f(x^*)+\tau(f(x^0)-f(x^*)), 
\end{equation}
where, for each of the $66$ problems, $x^N$ denotes the best point found by an algorithm after $N$ function calls to the noisy objective $f_\Theta$, $x^*$ is the best known solution and $\tau\in [0,1]$ is the convergence tolerance. Thus, a problem is said to be solved within the convergence tolerance~$\tau$ if~\eqref{fNacc} holds.}

The horizontal axis of the data profiles shows the number of noisy function evaluations divided by $n+1$ while the vertical axis shows the portion of problems solved within a given convergence tolerance~$\tau$. {\blc The horizontal axis of the performance profiles shows the ratio of the number of function calls to the noisy blackbox while the vertical axis shows the portion of problems solved within the tolerance~$\tau$.} In all the experiments, a budget of $1000(n+1)$ noisy function evaluations is set, i.e all algorithms stop as soon as the number of function calls to $f_\Theta$ reaches $1000(n+1)$. For the initialization, the same common parameters to both methods are used: $\delta_m^0=\delta_p^0=1$ and {\blr the mesh refining parameter} $\tau=1/2$. StoMADS parameters $\gamma$ and $\ef$ are chosen {\blc arbitrarily} so that $\gamma\ef=0.17$. However, for the choice of the sample {\blr size} $p^k$, it is worthwhile to mention that {\sf NOMAD-robust} is not in line with the theory analyzed in this work, especially in term of {\blr sample} sizes which are not involved in its theory. Indeed, the blackbox is evaluated by {\sf NOMAD-robust} at each point only once, while it needs to be evaluated at least $p^k$ times by StoMADS at each point in order to construct the estimates $\fok=\frac{1}{p^k}\sum_{i=1}^{p^k}f_{\theta_{1,i}}(x^k)\approx f(x^k)$ and $\fsk=\frac{1}{p^k}\sum_{i=1}^{p^k}f_{\theta_{2,i}}(x^k+s^k)\approx f(x^k+s^k)$, where $\theta_{1,i}$ and $\theta_{2,i}$,  $i\in\{1,2,\dots,p^k\}$, are the realizations, respectively, of the random variables $\Theta_{1,i}$ and $\Theta_{2,i}$ introduced in Section~\ref{computation}. 

This latter remark, in addition to the need for $p^k$ to be large in order for the estimates to be sufficiently accurate, therefore yields the following challenge that has to be faced: obtaining satisfactory solutions with the allocated budget, but requiring only {\blr few evaluations} of the stochastic blackbox during the estimates computation. {\blr Recall that $n^k$ denotes the number of blackbox evaluations at a given point when constructing an estimate at the iteration $k$.} Five variants of StoMADS corresponding respectively to $n^k=1,2,\dots,5$ for all $k$, are {\blr therefore} compared to {\sf NOMAD-robust}, {\sf NOMAD-basic} and {\sf NOMAD-default}, and despite the fact that the resulting values of $p^k$ do not meet the theoretical prescription derived in Section~\ref{computation}, they seemed to work well enough compared to many various other choices of $n^k$ that have been tested. {\blr However}, in order to increase the estimates accuracy while using few blackbox evaluations, {\blr the following procedure described in Section~\ref{computation} is used. Recall that it improves the estimates accuracy by making use of available samples at the current iterate during estimates computation, thus avoiding additional blackbox evaluations.
When the iteration $k$ is successful, the estimate $f_0^{k+1}$ of $f(x^{k+1})$ is computed according to~\eqref{comp1},
while after an unsuccessful iteration $k$, $f_0^{k+1}$ is given by~\eqref{comp2}.}

The three levels of noise that are considered in the experiments correspond respectively to $\sigma=1\%$, $\sigma=3\%$ and $\sigma=5\%$. {\blu These values are arbitrarily chosen in order to study how the portion of problems solved by StoMADS varies with the noise level}. Considering for example the Rosenbrock~\cite{cutest} test function given by 
\begin{equation}\label{rosenbrock}
f(x)= 100(x_2-x_1^2)^2+(1-x_1)^2,
\end{equation}
with the starting point ${\blr x^0}=(-1.2,1)$ and the minimum value $f^*=0$, then $\abs{f({\blr x^0})-f^*}=24.2$ and the corresponding noisy function is given by 
\begin{equation}\label{noisyrosenbrock}
f_\Theta(x)= \left[10(x_2-x_1^2)+\Theta_1\right]^2+\left[(1-x_1)+\Theta_2\right]^2,
\end{equation}
where $\Theta_1$ and $\Theta_2$ are independent random variables uniformly generated in the interval ${\blu I(\sigma,{\blr x^0},f^*)=}$ $\left[-24.2\sigma,24.2\sigma\right]$. Figure~\ref{rosenfig} shows the plots of the Rosenbrock function and its corresponding noisy versions. {\blc Figure~\ref{dataprof1pc},~\ref{dataprof3pc},~\ref{dataprof5pc} and Figure~\ref{perfprof1pc},~\ref{perfprof3pc},~\ref{perfprof5pc} present the data profiles and the performance profiles which compare the five variants of StoMADS with {\sf NOMAD-robust}, {\sf NOMAD-basic} and {\sf NOMAD-default} for various noise levels and convergence tolerances.

The data profiles and the performance profiles show in general that StoMADS outperforms not only {\sf NOMAD-robust}, but also both deterministic blackbox optimization algorithms {\sf NOMAD-basic} and {\sf NOMAD-default} which are obviously not appropriate for stochastic optimization. Moreover, changing the value of the tolerance parameter~$\tau$ in the performance profiles does not significantly alter the conclusions drawn from the data profiles. {\blr Thus, it can be noticed that for a given~$\tau$, the higher the noise level, the lower is the portion of problems solved for most variants of StoMADS as expected. Indeed, since the variance of the noise in the noisy blackbox augments with the noise level, it follows from Section~\ref{computation} that the estimates need to be sufficiently accurate to generate satisfactory solutions and consequently allow the resolution of a larger portion of problems. Similarly, for a fixed noise level, the higher the convergence tolerance, the larger is the portion of problems solved by most algorithms.

Furthermore, even though the number $n^k$ of blackbox evaluations is constant from one iteration to another for a given variant of StoMADS, this is not the case for the sample size $p^k$ involved in the estimates computation. Indeed, it follows respectively from~\eqref{comp1} and~\eqref{comp2} that $p^{k+1}=2n^k$ when the iteration $k$ is successful while $p^{k+1}=p^k+n^{k+1}$ when it is unsuccessful. Thus, even though the efficiency of each StoMADS variant depends on its corresponding evaluation parameter $n^k$, the quality of the solutions that are generated is influenced by the sample rate $p^k$ which is not constant. This explains why varying the blackbox evaluation parameter $n^k$ from one to five does not necessarily improve the performance of the corresponding StoMADS variants. Note that this also explains why the behavior of the StoMADS variant corresponding to $n^k=1$ is not similar to that of MADS. Indeed, no estimates computation is carried out in MADS and moreover, MADS is unable to show how an improvement in a noisy blackbox can lead to a decrease in an available objective function unlike StoMADS.} 

It follows from these results, specifically the analysis of the profiles corresponding to the tolerance $\tau=10^{-3}$, that StoMADS can handle the optimization of stochastically noisy blackboxes that are expensive in term of blackbox evaluations, {\blc since its variants corresponding to $n^k=1$ and $n^k=2$ are able to generate satisfactory solutions thus using few blackbox evaluations}. However, the choice $n^k=4$ seems to be preferable for stochastic blackbox optimization problems with higher evaluations budgets.} 


\begin{figure}[ht!]
\centering
\includegraphics[scale=0.456]{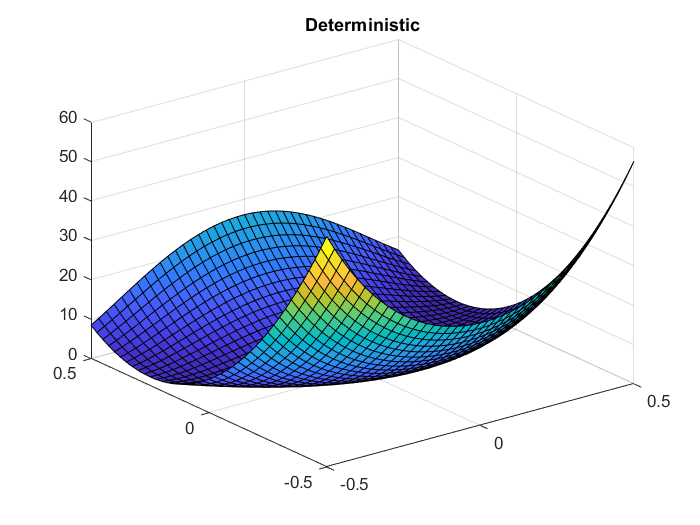}   
\includegraphics[scale=0.456]{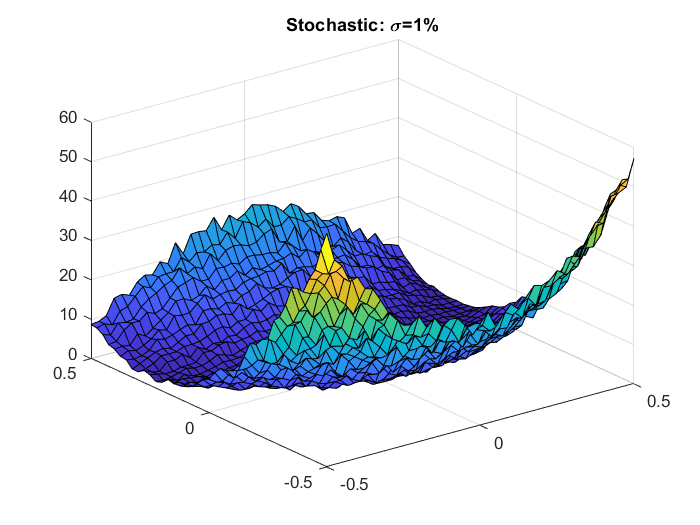} 
\includegraphics[scale=0.456]{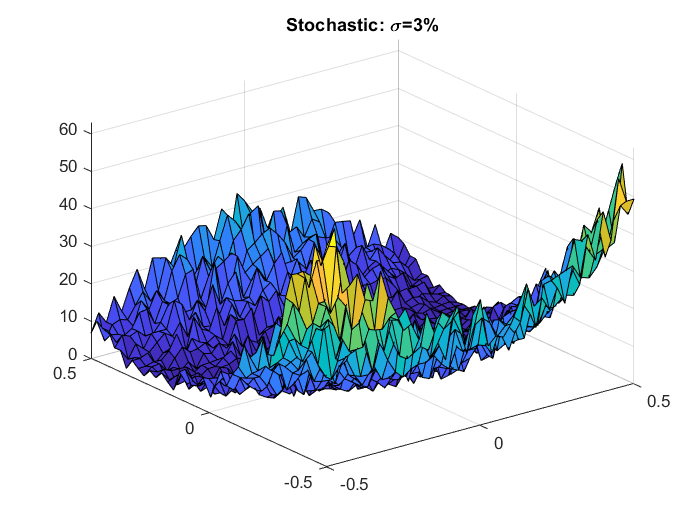}
\includegraphics[scale=0.456]{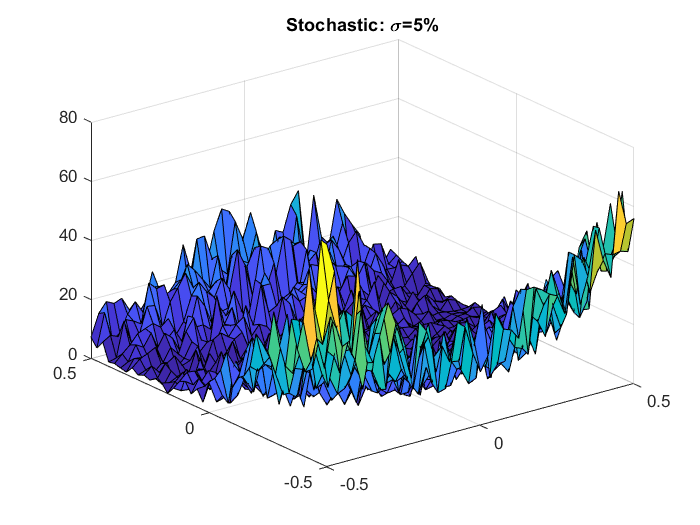}
\caption[]{\small{Plots of the deterministic Rosenbrock function~\eqref{rosenbrock} and its corresponding noisy versions~\eqref{noisyrosenbrock} on the box $[-0.5,0.5]\times [-0.5,0.5]$. The random variables defining the noisy functions $f_\Theta$ are uniformly generated in $[-24.2\sigma, 24.2\sigma]$.}}
\label{rosenfig}
\end{figure}
\clearpage
\twocolumn

\begin{figure}[]
\centering
\includegraphics[scale=0.49]{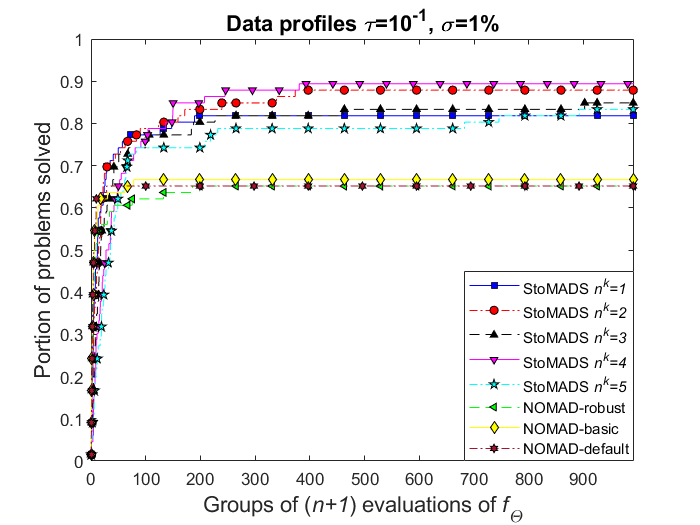}       
\includegraphics[scale=0.49]{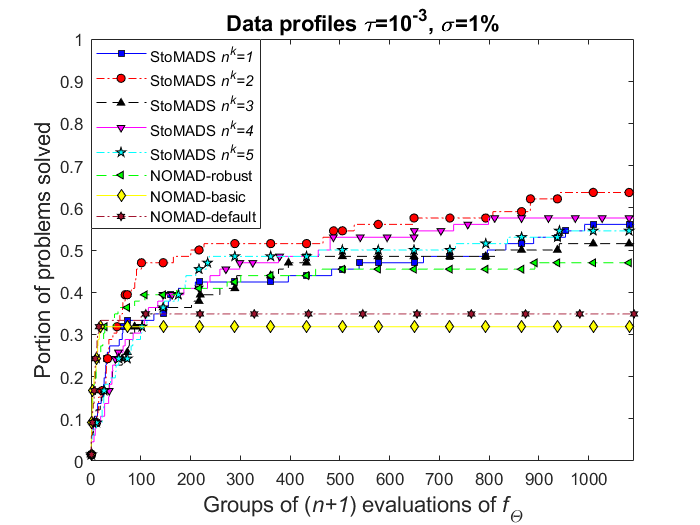}
\caption[]{\small{Data profiles  for noise level $\sigma=1\%$ and convergence tolerances $\tau=10^{-1}$ and $\tau=10^{-3}$ on 66 analytical unconstrained test problems additively perturbed in the interval $I(\sigma,{\blr x^0},f^*)$.}}
\label{dataprof1pc}
\end{figure}
\begin{figure}[]
\centering
\includegraphics[scale=0.49]{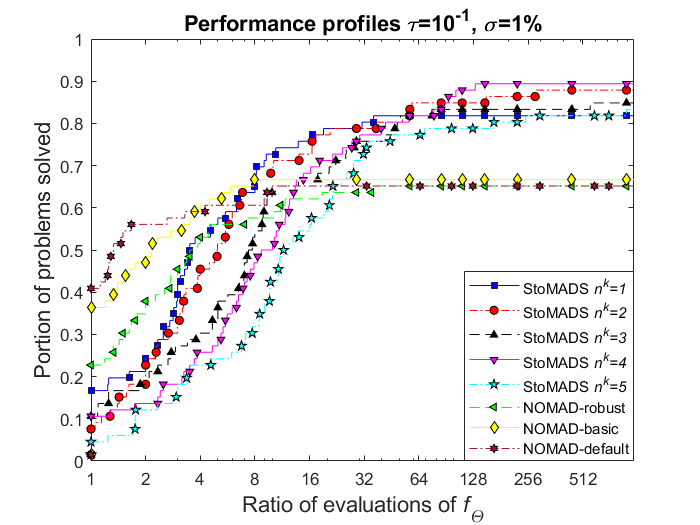} 
\includegraphics[scale=0.49]{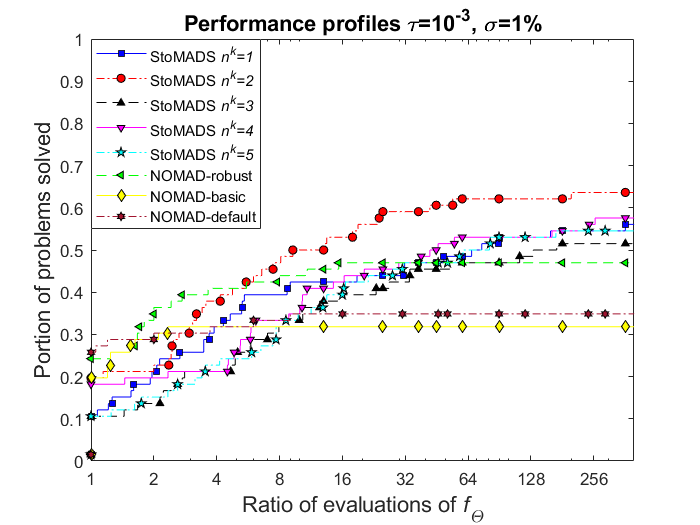}
\caption[]{\small{Performance profiles for noise level $\sigma=1\%$ and convergence tolerances $\tau=10^{-1}$ and $\tau=10^{-3}$ on 66 analytical unconstrained test problems additively perturbed in the interval $I(\sigma,{\blr x^0},f^*)$.}}
\label{perfprof1pc}
\end{figure}

\clearpage

\begin{figure}[]
\centering
\includegraphics[scale=0.49]{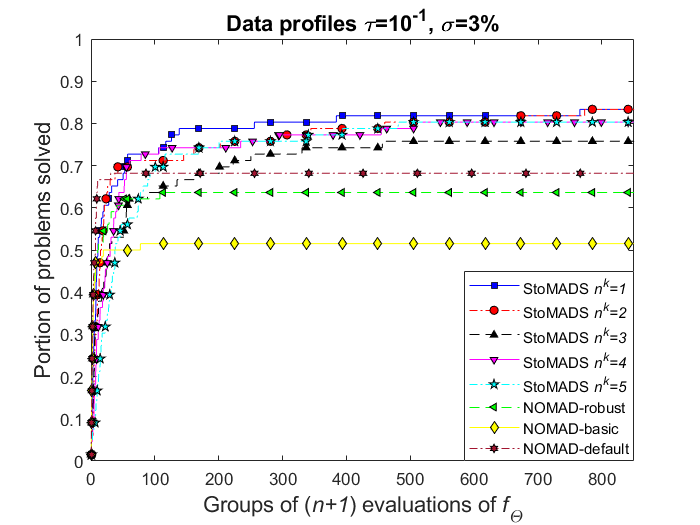}
\includegraphics[scale=0.49]{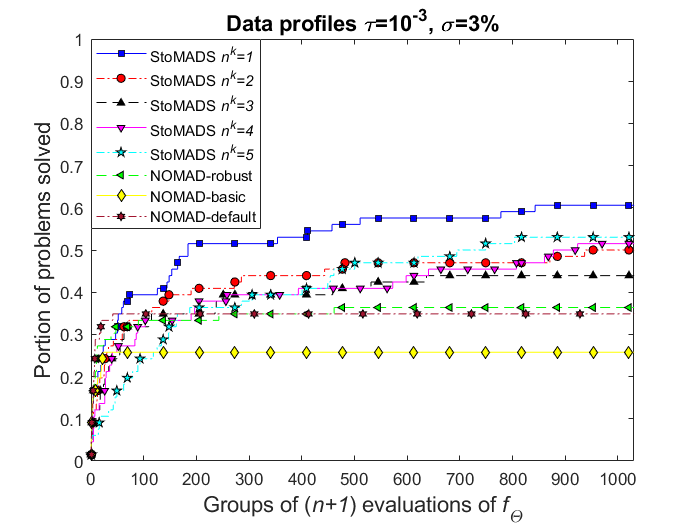}
\caption[]{\small{Data profiles  for noise level $\sigma=3\%$ and convergence tolerances $\tau=10^{-1}$ and $\tau=10^{-3}$ on 66 analytical unconstrained test problems additively perturbed in the interval $I(\sigma,{\blr x^0},f^*)$.}}
\label{dataprof3pc}
\end{figure}

\begin{figure}[]
\centering
\includegraphics[scale=0.49]{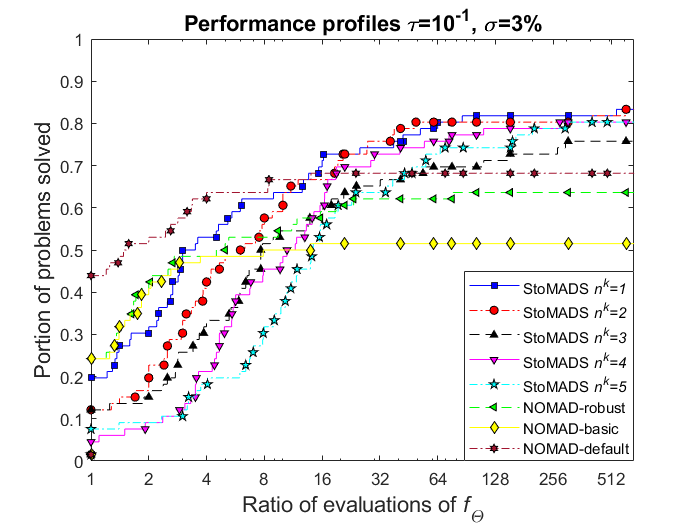}
\includegraphics[scale=0.49]{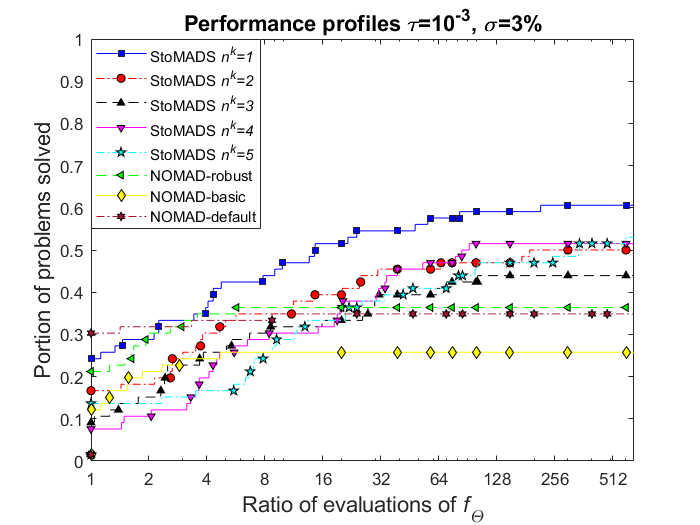}
\caption[]{\small{Performance profiles for noise level $\sigma=3\%$ and convergence tolerances $\tau=10^{-1}$ and $\tau=10^{-3}$ on 66 analytical unconstrained test problems additively perturbed in the interval $I(\sigma,{\blr x^0},f^*)$.}}
\label{perfprof3pc} 
\end{figure}

\clearpage

\begin{figure}[]
\centering
\includegraphics[scale=0.49]{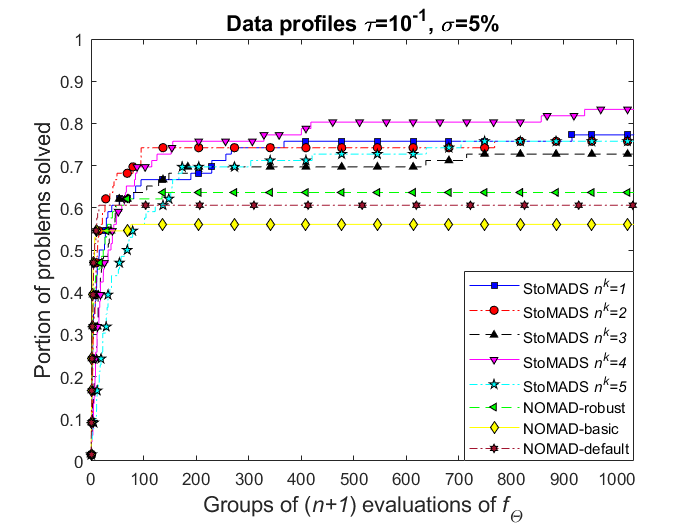}
\includegraphics[scale=0.49]{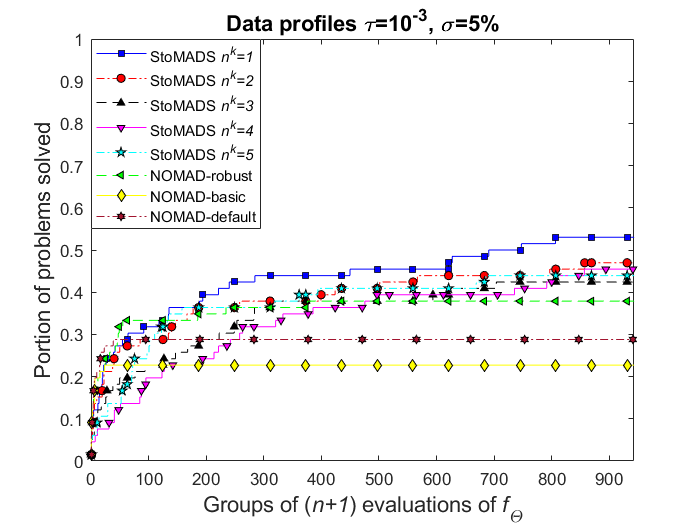}
\caption[]{\small{Data profiles  for noise level $\sigma=5\%$ and convergence tolerances $\tau=10^{-1}$ and $\tau=10^{-3}$ on 66 analytical unconstrained test problems additively perturbed in the interval $I(\sigma,{\blr x^0},f^*)$.}}
\label{dataprof5pc} 
\end{figure}

\begin{figure}[]
\centering
\includegraphics[scale=0.49]{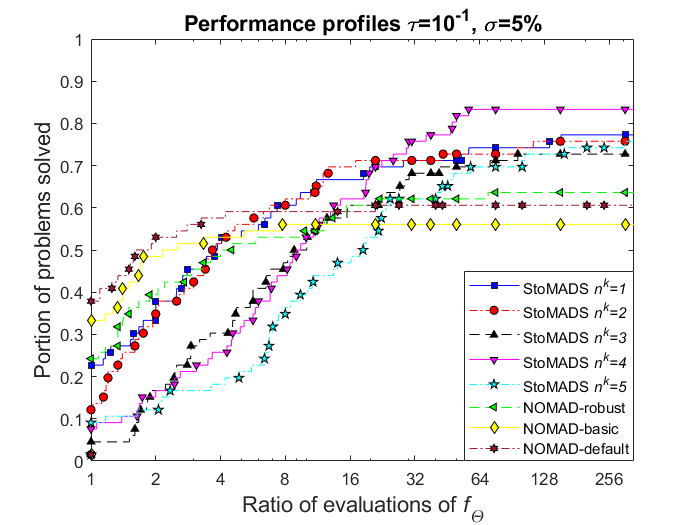}
\includegraphics[scale=0.49]{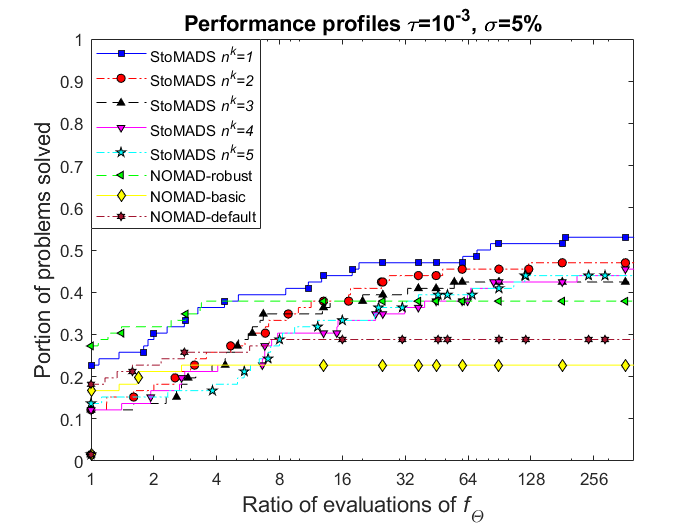}
\caption[]{\small{Performance profiles for noise level $\sigma=5\%$ and convergence tolerances $\tau=10^{-1}$ and $\tau=10^{-3}$ on 66 analytical unconstrained test problems additively perturbed in the interval $I(\sigma,{\blr x^0},f^*)$.}}
\label{perfprof5pc} 
\end{figure}

\clearpage
\onecolumn

\section*{\blr Discussion}
MADS is a valuable blackbox optimization algorithm with full-supported convergence analysis, but it is designed for deterministic problems. Even though Robust-MADS, the first variant of MADS designed for noisy blackbox optimization, was shown to have {\it {\blu zero}-order} convergence properties, the corresponding work~\cite{AudIhaLedTrib2016} did not show how an improvement in the smoothed version of the noisy available blackbox, used to update the iterates, should result in a decrease in the unknown objective.
	
{\blr Thus, unlike Robust-MADS, the method proposed in this manuscript, StoMADS, clearly shows how an improvement in the estimates of the unavailable objective function values may cause a decrease in the unavailable objective function. This is achieved by defining new iteration types by means of a sufficient decrease condition on these estimates that are required to be probabilistically sufficiently accurate.}

Although the convergence analysis of StoMADS uses ideas derived from that of MADS, the analysis itself is different and based on stochastic processes theory. In addition to the convergence result of the whole sequence of random mesh size parameters, which is stronger than the $\liminf$-type result of MADS, a more general existence proof of {\it refining} subsequences consisting of StoMADS iterates that are not necessarily mesh local optimizers has been proposed, followed by a stochastic variant of the Clarke optimality result of MADS.

An extensive computational study of several variants of StoMADS on a collection of unconstrained stochastically noisy problems shows that the proposed method outperforms Robust-MADS and also highlights the fact that MADS is not appropriate for stochastic blackbox optimization, even though StoMADS  estimates accuracy do not meet the prescription that has been derived theoretically.
	
Note that compared to all prior works using a theory similar to the one analyzed in this manuscript, the present research is to the best of our knowledge the first that requires no model or gradient information to find descent directions.
	
Future research will focus on extending this approach to stochastically noisy constrained and/or chance constraints blackbox optimization.

\section*{Acknowledgments}
{\blc The authors are grateful to Erick Delage from HEC MontrÈal and Richard Labib from Polytechnique MontrÈal for valuable discussions and constructive suggestions. This work is supported by the NSERC CRD RDCPJ 490744-15 grant and by an Innov\'E\'E grant, both in collaboration with Hydro-QuÈbec and Rio Tinto.}

\clearpage

\bibliographystyle{plain}
\bibliography{bibliography}

\end{document}